\documentclass[10pt]{amsart}
\usepackage{amsmath,amsthm,amsfonts,amssymb}
\usepackage{hyperref}
\usepackage[all]{xy}
\usepackage{graphicx}
\usepackage{amscd}
\usepackage{pb-diagram}
\usepackage{color}
\usepackage[mathscr]{eucal}
\usepackage{tikz}
\usetikzlibrary{patterns}

\usepackage{enumitem}

\numberwithin{equation}{section}

\newtheorem{theorem}{Theorem}[section]

\newtheorem{corollary}[theorem]{Corollary}

\newtheorem{lemma}[theorem]{Lemma}
\newtheorem{remark}[theorem]{Remark}
\newtheorem{prop}[theorem]{Proposition}

\newcommand\R{{\mathbb R}}
\newcommand\Z{{\mathbb Z}}
\newcommand\bS{{\mathbb S}}

\newcommand\sS{{\mathscr S}}

\newcommand\BI{{\mathbf I}}

\newcommand\CB{{\mathcal B}}
\newcommand\CC{{\mathcal C}}
\newcommand\CDD{{\mathcal D}}
\newcommand\CF{{\mathcal F}}
\newcommand\CR{{\mathcal R}}
\newcommand\CS{{\mathcal S}}

\newcommand\CI{{\mathcal I}}
\newcommand\CM{{\mathcal M}}
\newcommand\CN{{\mathcal N}}
\newcommand\wh[1]{{\widehat{#1}}}
\newcommand\wt[1]{{\widetilde{#1}}}

\newcommand\FB{{\mathfrak B}}
\newcommand\FD{{\mathfrak D}}

\newcommand\fp{{\mathfrak p}}

\newcommand\supp{\mbox{supp}}

\begin{document}

\author[E. Jeong]{Eunhee Jeong}
\author[S. Lee]{Sanghyuk Lee}

\address{School of Mathematics, Korea Institute for Advanced Study, Seoul 02455, Republic of Korea} \email{eunhee@kias.re.kr}

\address{Department of Mathematical Sciences and RIM, Seoul National University, Seoul 08826, Republic of  Korea}
\email{shklee@snu.ac.kr}

\title[Maximal estimates for bilinear operators]{Maximal estimates for   the bilinear spherical averages  and  the bilinear Bochner-Riesz operators}


\maketitle
%
%
%
\begin{abstract}
We study the maximal estimates  for the bilinear spherical  average 
and the bilinear Bochner-Riesz operator. First, we obtain $L^p\times L^q \to L^r$ estimates for  the bilinear spherical maximal function on the optimal range. 
Thus,  we settle the problem which was previously considered by 
 Geba, Greenleaf,  Iosevich, Palsson and Sawyer,  later  Barrionevo, Grafakos, D. He, Honz\'ik and Oliveira, and recently Heo, Hong and Yang.   
 Secondly,  we consider $L^p\times L^q \to L^r$ estimates for  the maximal bilinear Bochner-Riesz  operators and improve the previous known results. 
 For the purpose we draw a connection between the  maximal  estimates and  the square function estimates for the classical Bochner-Riesz operators.  
\end{abstract}
\section{Introduction and main theorems}

Let $d\ge 2$ and $m$ be a measurable function on $\R^d\times \R^d$.  The bilinear multiplier operator $T_m$ associated with $m$ is defined by 
\[ T_m(f,g)(x)=\int_{\R^d\times \R^d} e^{2\pi i x\cdot (\xi+\eta)} m(\xi,\eta) \wh f(\xi)\wh g(\eta) d\xi d\eta,\]
for Schwartz functions $f$ and $g$ in $\CS(\R^d)$. The study  on  $L^p\times L^q \to L^r$ boundedness of $T_m$  has a long history. 
   After appearance of  the seminal work of Lacey and Thiele \cite{LT1, LT2} on the boundedness of the bilinear Hilbert transform, there have been attempts to extend the earlier results to  the 
bilinear multiplier operators with less regular $m$. We refer the interested reader to \cite{
CoM,m-t,tomita} and references therein for more on background and related results.
In this note we are concerned with maximal bounds on the bilinear counterparts of  a couple notable operators, the bilinear spherical average and  the bilinear Bochner-Riesz operator.

\noindent \emph{The bilinear spherical maximal function.}
Let $d\ge2$ and $d\sigma_{d-1}$ be the induced surface measure on the sphere $\mathbb S^{d-1}$ in $\R^d$. The spherical maximal function 
\begin{equation*}\label{sm}
\sS f(x) =\sup_{t>0} \int_{\mathbb S^{d-1}}|f(x-ty)|d\sigma_{d-1} ,
\end{equation*}
was first studied  by Stein in \cite{stein}. He showed that, for $d\ge3$, $\|\sS f\|_{L^p(\R^d)}\le C\|f\|_{L^p(\R^d)}$ holds if and only  if   $\frac d{d-1}<p\le \infty$. 
By considering a suitable input function $f$, Stein also showed that  $L^p$ boundedness of $\sS$ fails for $p\le \frac d{d-1}$ and $d\ge2$. 
The case  $d=2$ turned out to be much more difficult than the problem for $d\ge3$. This is due to the fact that  the classical strategy based on $L^2$ estimate does not work since   $\sS$ is unbounded on $L^2(\R^2)$. The remaining case was later obtained by  Bourgain \cite{B}. 
Afterward,  Mockenhaupt-Seeger-Sogge provided a new proof for Bourgain's  result which relies on the local smoothing estimate for the wave operator \cite{MSS}.

In the first part of this paper, we mainly discuss  $L^p\times L^q\to L^r$ boundedness of a bilinear analogue  of $\sS$.
The  \emph{bilinear spherical maximal function}  $\CM$ is defined by
\begin{equation*}\label{bsm}
{\CM}(f,g)(x)= \sup_{t>0}  \int_{\bS^{2d-1}}| f(x-ty)g(x-tz) | d\sigma_{2d-1}(y,z).
\end{equation*}
The operator $\CM$  first appeared in \cite{GGIPES} and,  subsequently,  studied by   Barrionevo-Grafakos-D.He\footnote{To avoid possible confusion,  we add his first initial.}-Honz\'ik-Oliveira  \cite{BGHHO}, Grafakos-D.He-Honz\'ik \cite{GHH1}, and Heo-Hong-Yang \cite{HHY}.  Let $1\le p,q\le \infty$ and $0<r\le\infty$ satisfy the H\"older relation 
\begin{equation}
\label{holder}
\frac1p+\frac1q=\frac1r.
\end{equation} In a recent work  \cite{BGHHO}, Barrionevo, Grafakos, D.He, Honz\'ik, and Oliveira  showed that 
\begin{equation}\label{eq:pqr}
\|\CM(f,g)\|_{L^r(\R^d)}\lesssim \|f\|_{L^p(\R^d)}\|g\|_{L^q(\R^d)}
\end{equation}
holds when $(\frac1p,\frac1q)$ is in a open triangle with vertices $A=(1,0)$, $B=(0,1)$, $O=(0,0)$ for $d\ge 2$, and in a open quadrilateral with vertices $A,$ $C=(\frac{2d-10}{2d-5},\frac{2d-10}{2d-5})$, $B,$ $O$ for $d\ge 8$. (See Figure \ref{fig}). For $1\le p,q,r\le \infty$ (the Banach triangle case), they utilized the  boundedness of linear maximal operators which are associated with multipliers of limited decay (see  Rubio de Francia  \cite{ru}). To extend the range of exponents outside the Banach triangle, they obtained $L^2\times L^2\to L^1$ bound for $\CM$ via a wavelet decomposition.  
In \cite{GHH1}, Grafakos, {D.\,He}, and Honz\'ik obtained a bilinear analogue of Rubio de Francia's result in \cite{ru}  and as its application they obtained the estimate \eqref{eq:pqr} for $p=q=2$ when $d\ge4$. The result in \cite{GHH1} was very recently improved by Heo, Hong, and Yang \cite{HHY}. Their argument relies on a decomposition and the asymptotic expansion of the Fourier transform of $d\sigma_{2d-1}$. More precisely, they proved the estimate \eqref{eq:pqr} for $(\frac1p,\frac1q)$ which is contained in the open hexagon with vertices $A$, $D=(\frac{d-3}{d-2},1)$, $E=(\frac{2d-2}{2d-1},\frac{2d-2}{2d-1})$, $F=(1,\frac{d-2}{d-2})$, $B$ and $O$ when $d\ge 4$. They also obtained the estimate \eqref{eq:pqr} in an open hexagon including the Banach triangle when $2\le d\le 3$. (See \cite[Theorem 1]{HHY}).

\begin{figure}\label{fig}
\begin{tikzpicture}
\draw (0,0) rectangle (5,5);
\path[fill=gray!35] (0,0)--(0,5)--(4,5) --(5,4)--(5,0);
\draw[dash pattern= { on 2pt off 1pt}]  (4,5)--(5,4); 
\draw (4,5)--(0,5)--(0,0)--(5,0)--(5,4);
\node [below left] at (0,0) {$O$};\node [above] at (0,5) {$A$};
\node [above] at (4.2,5) {$G=(\frac{d-1}d,1)$};
\draw (4,5) circle [radius=0.04];
\node [right] at (5,4) {$H=(1,\frac{d-1}d)$};
\draw (5,4) circle [radius=0.04];
\node [below right] at (5,0) {$B$};
\node [ right] at (3,3) {C};
\draw (3,3) circle [radius=0.04]; \draw (5,0) circle [radius=0.04]; \draw (0,5) circle [radius=0.04];
\draw[dash pattern={on 2pt off 1pt}] (5,0)--(3,3)--(0,5); 
\draw[dash pattern={on 2pt off 1pt}] (0,0)--(5,5); 
\draw[dash pattern={on 2pt off 1pt}]  (3.05,5)--(4.2,4.2)--(5,3.05);
\draw (3.05,5) circle [radius=0.04]; \node[below left] at (3.05,5) {$D$}; 
\draw (5,3.05) circle [radius=0.04]; \node[below left] at (5,3.05) {$F$}; 
\draw (4.2,4.2) circle [radius=0.04]; \node [left] at (4.2,4.2) {$E$}; 
\draw [<->] (5.8,0.7)--(5.8,0)--(6.5,0);
\node at (6.7,0) {$\frac1p$};
\node at (5.8,0.7) [right]{$\frac1q$};
\end{tikzpicture}
\caption{$L^p\times L^q\to L^r$ boundedness of $\CM$, $d\ge 2$.}
\end{figure}

The following is our first result which completely characterizes $p, q,r$ for which \eqref{eq:pqr} holds. 

\begin{theorem}\label{thm:bsm} 
Let $d\ge2$. Let $1\le  p,q\le \infty$ and $0<r\le \infty$.  Then, the estimate \eqref{eq:pqr}  holds if and only if  $r>\frac d{2d-1}$ and  \eqref{holder} holds except the case $(p,q,r)=(1,\infty,1)$ or $(\infty, 1,1)$.
In addition, we have weak estimates in terms of Lorentz spaces:
\begin{equation}
\label{lorentz}
 \|\CM(f,g)\|_{L^{r,u}(\R^d)}\lesssim \|f\|_{L^{p,s}(\R^d)}\|g\|_{L^{q,t}(\R^d)}.
 \end{equation} 
 \begin{itemize}
 [leftmargin=0.28in]
 \item[$(\bf a)$]
    If  $p=r=1$ and \eqref{holder} holds, then  \eqref{lorentz} holds with $u=t=\infty$ and $s=1$.  
\end{itemize}
\vspace{-5pt}
Additionally, if $d\ge 3$, we have the following: 
\vspace{-10pt}
 \begin{itemize}
  [leftmargin=0.28in]
 \item[$(\bf b)$]
    If $p=1$, $q=\frac d{d-1}$, and \eqref{holder} holds, then  \eqref{lorentz} holds with $u=\infty$ and $s=t=1$. 
\item[$(\bf c)$]   
 If $1<p<\frac{d}{d-1}$, $ r=\frac d{2d-1}$, and \eqref{holder} holds, then \eqref{lorentz} holds with $u=\infty$ and  $s,t$ satisfying $\frac1s+\frac1t=\frac{2d-1}d$ and $s,t>0$.  
\end{itemize}
The assertions ${\bf(a)}$--$({\bf c})$ are also true when the roles of $(p,s)$ and $(q,t)$ are interchanged.
\end{theorem}

Actually,  we obtain estimates for a stronger maximal operator, see Remark \ref{strong}. 
Necessity of the condition \eqref{holder} and $r>\frac d{2d-1}$ is easy to see. 
Indeed,  since \[{\CM}(f(R\cdot), g(R\cdot))(x/R)={\CM}(f,g)(x), \  \forall R>0,\] 
 by scaling one can easily see that   $\|\CM (f,g)\|_r/(\|f\|_p\|g\|_q)\sim  R^{d(\frac1r-\frac1p-\frac1q)}$ for any $R>0$ as long as the estimate \eqref{eq:pqr} is true.
And it was shown in \cite{BGHHO} that the estimate \eqref{eq:pqr} fails when $r\le \frac{d}{2d-1}.$  This was done by testing  variants of the function against  \eqref{eq:pqr}  which  was used by Stein to show the sharp range of boundedness of the spherical maximal function.  The failure of  \eqref{eq:pqr} for  $r< \frac{d}{2d-1}$ can also be verified by  a simpler Knapp type example (see Proposition  \ref{pro:low1}).

Our result is based on a simple observation that  $\CM$ can be bounded  by a product of the Hardy-Littlewood maximal function $M$ and the (linear) spherical maximal function $\sS$ (see Lemma \ref{lem:pt}). This is done by  a kind of slicing argument.  As a consequence, Theorem \ref{thm:bsm} is verified by making use of the known bounds for $M$ and $\sS$, and interpolation.  The argument which we use to  prove Theorem \ref{thm:bsm}  continues to work for  the general $k$--linear maximal operator. See Remark \ref{multi} for more details.

The estimates for $p,q$ satisfying $\frac1r=\frac1p+\frac1q =\frac{2d-1}d$ are of special interest and it seems likely that  the weak estimates in this case can be further improved. These estimates correspond to the critical endpoint estimate for $\sS$ with $p=\frac d{d-1}$. For the spherical maximal function  Bourgain \cite{B0} showed the estimate $\|\sS f\|_{L^{\frac{d}{d-1},\infty}}\lesssim \|f\|_{L^{\frac{d}{d-1},1}}$  when $d\ge 3$ but failure of such estimate when $d=2$ was shown by Seeger, Tao, and Wright \cite{STW}.

\subsubsection*{Localized  maximal function} Let us consider the maximal operator $\wt\sS$ which is given by 
taking supremum over $t\in[1,2]$: 
 \begin{equation*}\label{tu_sm}
\wt\sS f(x) =\sup_{1\le t\le 2} \int_{\mathbb S^{d-1}}|f(x-ty)|d\sigma_{d-1}(y).
\end{equation*}
Though the estimate for $\wt\sS$ looks weaker than that for $\sS$, by using the Littlewood-Paley theory (for example, see \cite{schlag0}), one can deduce the $L^p$-bound for $ \sS$  from the estimate for the truncated maximal operator $\wt{\sS}$. 
Thanks to the localization in $t$,  $f\mapsto \wt\sS f $ exhibits $L^p$-improving properties, that is to say,  
\begin{equation}
\label{improving}
\|\wt\sS f\|_v\le C \|f\|_u
\end{equation}  with some $u<v$.  This was observed 
in the work of  Mockenhaupt-Seeger-Sogge \cite{MSS} with $d=2$.  Later on,  Schlag \cite{schlag} characterized almost complete  
set of $(u,v)$ for which \eqref{improving} holds when $d=2$. 
Schlag and Sogge \cite{SS} extended such result to the higher dimensions, $d\ge 3$, but the estimates on the borderline  were missing. One of the authors \cite{Lee} obtained most of the (left open) endpoint estimates for $\wt\sS$ on the borderline for $d\ge 2$ but  there are  still a few endpoint  estimates of which validities are not settled yet. See Theorem \ref{local_sm} below.

As in the linear case, we consider a localized bilinear maximal function $\wt\CM$ given by 
\begin{equation}\label{tu_bsm}
\wt{\CM}(f,g)(x)= \sup_{1\le t\le 2} \Big| \int_{\bS^{2d-1}} f(x-ty)g(x-tz) d\sigma_{2d-1}(y,z)\Big|.
\end{equation}
Thanks to the localization of $t$,  the  operator $(f,g)\mapsto \wt{\CM}(f,g)$ is 
free of scaling invariance.  Thus, it is natural to expect that 
\begin{equation*}\label{loc_pqr} 
\|\wt\CM(f,g)\|_{L^r(\R^d)}\le C \|f\|_{L^p(\R^d)}\|g\|_{L^q(\R^d)}
\end{equation*}
holds true on a wider range of  $p,q,r$ which do not necessarily satisfy the H\"older relation \eqref{holder}. In particular, we manage to obtain  the sharp  range of exponents $p,q$ when $r$ is in a certain region.
See Section \ref{sec:loc}.

\noindent\emph{The maximal bilinear Bochner-Riesz operator.}
We now consider the bilinear Bochner-Riesz operator $\CB^\alpha_\lambda$ of order $\alpha \ge 0$, which is a bilinear multiplier operator defined by 
\begin{equation*}\label{bBR}
\CB^\alpha_\lambda(f,g)(x)=\iint_{\R^d\times\R^d} e^{2\pi i x\cdot(\xi+\eta)}\big(1-|\lambda\xi|^2-|\lambda\eta|^2\big)^\alpha_+\,\wh{f}(\xi)\wh g(\eta)d\xi d\eta,\quad \lambda>0,
\end{equation*}
for $f$ and $g$ in $\CS(\R^d)$. Here, $r_+ =r$ for $r>0$ and $r_+=0$ for $r\le 0$, and $\widehat f$ is the Fourier transform of $f$ given by $\int_{\R^d} e^{-2\pi i x\cdot \xi}f(x)dx$. The bilinear Bochner-Riesz operator is not only a model operator of which multiplier has singularities with non-vanishing Gaussian curvature, but also a natural bilinear extension of the classical Bochner-Riesz operator $\CR^\alpha_\lambda$ which is given by
\[ \CR^\alpha_\lambda(f)(x) =\int_{\R^d} e^{2\pi i x\cdot \xi} (1-|\lambda\xi|^2)^\alpha_+ \,\widehat{f}(\xi) d\xi,\quad f\in \CS(\R^d).\]
The study of the Bochner-Reisz operator has its origin at understanding  summability of the Fourier series. Especially, related to norm convergence of the Fourier series,  $L^p$-boundedness of  the Bochner-Reisz operator has been studied. The conjecture, which is known as the Bochner-Riesz conjecture, is {stated} as follows:   For $1\le p\le \infty$ except $p=2$, the operator $\CR^\alpha_\lambda$ is bounded on $L^p(\R^d)$  if and only if 
\begin{equation*}\label{critical}
 \alpha>\alpha(p) =\max\Big\{ d\Big|\frac12-\frac1p\Big|-\frac12, 0 \Big\}.
 \end{equation*}
This  was proved by Carleson and Sj\"olin \cite{CSj} for $d=2$. For $d\ge3$  substantial progresses have  been achieved for the last couple of decades but the conjecture still remains open. We refer to \cite{Fe1,B,TV,lee1,BGu,GHiI} and references therein for details. The maximal operator $\CR^\alpha_*=\sup_{\lambda>0}|\CR^\alpha_\lambda|$ has been of interest in its connection to pointwise convergence of the Fourier series, and the maximal estimate may also be regarded as a vector valued generalization of  the estimate for   $\CR^\alpha_\lambda$. It is also conjectured that for $p>2$ the maximal operator $\CR^\alpha_*$ is bounded on $L^p(\R^d)$ if and only if $\alpha>\alpha(p)$. For $d=2$ the conjecture was shown to be true by Carbery \cite{Ca1}. For $d\ge 3$  partial results are known  although the corresponding pointwise convergence with the optimal order was shown by Carbery-Rubio de Francia-Vega who used  $L^2$ weighted inequality \cite{ CaRuVe}. See \cite{Ch,S1, lee1,LS, An, lee2} and references therein for more details and recent results.   When $1<p<2$, it turned out that $L^p$-boundedness of $\CR^\alpha_*$  is different to that of  $\CR^\alpha_\lambda$. An additional necessary condition was shown by Tao \cite{tao}.  

 Recently, $L^p\times L^q\to L^r$ boundedness of the   bilinear multiplier operators (including that of the bilinear Bochner-Reisz operator $\CB^\alpha_\lambda$)  has been studied by several authors
 \cite{GLi,DiG,BeGe,BeGSY, JLV}. In particular, the authors of \cite{JLV} used a new idea which splits the interaction between two variables $\xi$ and $\eta$ in the Fourier side, and made a connection between the boundedness of $\CB^\alpha_1$ and the square function estimates for the (linear) Bochner-Riesz operator. Consequently, they managed to improve  the previous known results and obtained some sharp bounds when $d=2$.

The maximal estimates for $\CB^\alpha_\lambda$  were recently studied  in \cite{He,GHH1}.  Grafakos, D.He, and Honz\'ik \cite{GHH1} showed that the maximal operator $\CB^\alpha_*=\sup_{\lambda>0} |\CB^\alpha_\lambda|$ is bounded from $L^2(\R^d)\times L^2(\R^d)$ to $L^1(\R^d)$ for $\alpha>\frac{2d+3}4$. In \cite{GHH1},  $L^2\times L^2\to L^1$ boundedness was shown for general maximal operators $\sup_{t>0}|T_{m(t\cdot)}(f,g)|$ 
of which (bilinear) multiplier $m$ has a limited decay (see \cite[Theorem 1.1]{GHH1}). As an application, they obtained the aforementioned result  on the $L^2\times L^2\to L^1$ boundedness of $\CB^\alpha_*$. This estimate can be interpolated with other easier estimates to give  the $L^p\times L^q\to L^r$ bound for the other exponents $p,$ $q,$ and $r$ with some range of $\alpha$. However, these estimates seem to be far from being optimal. 

From now on, we set
\begin{equation*}\label{mbbr}
\CB^\alpha_*(f,g)=\sup_{\lambda>0}|\CB^\alpha_\lambda(f,g)|.
\end{equation*} 
The second half of this paper is devoted to improving  the range of $\alpha$ for which the maximal operator  $\CB^\alpha_\ast$ is bounded from $L^p(\R^d)\times L^q(\R^d)$ to $L^r(\R^d)$ when $p,q\ge 2$. Especially, we adopt the decomposition strategy in  Jeong-Lee-Vargas \cite{JLV} and draw a connection between boundedness of $\CB^\alpha_*$ and the square function estimates associated with the classical Bochner-Riesz operator. More precisely, for $0<\delta\ll 1$ and a smooth function $\varphi$ supported in $[-1,1]$,  we consider a square function $\mathfrak S_\delta^\varphi$ which is given by 
\begin{equation}\label{lo-square}
  \mathfrak{S}_\delta^\varphi f(x) = \Big(\int_{1/2} ^2\Big|\varphi\Big(\frac{t-|D|^2}\delta\Big) f(x)\Big|^2dt\Big)^{1/2}.
  \end{equation}
It is conjectured that for $s\ge\frac{2d}{d-1}$ and $\epsilon>0$,  there exists 
$C=C(\epsilon)$ such that 
\begin{equation}
\label{del-square}
 \|\mathfrak{S}_\delta^\varphi f\|_{L^s(\R^d)}\le C\delta^{  -\frac{d-2}2+\frac dp-\epsilon}\|f\|_{L^s(\R^d)}.
 \end{equation}  
The estimate  \eqref{del-square}  has been studied by many authors 
(\cite{Ca1, Ch, LRS, lee2}). The conjecture \eqref{del-square} not only implies the maximal Bochner-Riesz conjecture but also has various applications (see \cite{CGT, LRS} and references therein). The most recent result for the estimate \eqref{del-square} can be found in Lee-Rogers-Seeger \cite{LRS} and Lee \cite{lee2} (see Theorem \ref{LRS_L} below).

We now introduce some notations to present our result. For $\nu\in [0, \frac{d-1}{2d}]$, we set 
\begin{align*}
\CDD_1(\nu)=& \big \{(u,v)\in [0,1/2]^2 : u,v\le \nu \big \},
\quad \CDD_2(\nu)= \big \{(u,v)\in  [0,1/2]^2 :  u, v\ge \nu \big \},
\\
\CDD_3(\nu)=&
\big \{(u,v)\in  [0,1/2]^2 : u < \nu <v \mbox{ or } v < \nu< u\big \}.
\end{align*}
The regions $\CDD_j(\nu)$, $1\le j\le 3$, are pairwise disjoint and  $\bigcup_{j=1}^3 \CDD_j(\nu)=[0,1/2]^2$. 
We define a real valued function $\alpha_{1/\nu}^*: [2,\infty]^2\to \mathbb R$  by 
\begin{equation}\label{alpha}
\alpha^*_{1/\nu}(p,q)\!
= \! \begin{cases}      
\ \alpha(p)+\alpha(q)+1=  d(1-1/p-1/q), & \quad  (1/p,1/q)\in \CDD_1(\nu),  
\\
\  1+\frac{2(1-1/p-1/q)}{1-2\nu}\alpha(1/\nu),& \quad (1/p,1/q)\in \CDD_2(\nu), 
\\
\  1+ \alpha(p)\lor \alpha(q)+ \alpha(1/\nu) (\frac{p-2}{(1-2\nu)p}\land\frac{q-2}{(1-2\nu)q}),  & \quad   (1/p,1/q)\in \CDD_3(\nu).
\end{cases}    \end{equation}
Here $\alpha(p)=\max\{d|\frac1p-\frac12|-\frac12,0\}$ is the critical index for the $L^p$-boundedness of the  Bochner-Riesz operator. Our result for $\CB^\alpha_*$ is the following.

\begin{theorem}\label{thm:main}
 Let $d\ge 2$ and $ \fp\ge \frac{2d}{d-1}$. Let $2\le p,q\le\infty$ and $r$ satisfy $\frac1r=\frac1p+\frac1q.$  Suppose that for $s\ge \fp$ the estimate \eqref{del-square} holds with $C$ independent of $\varphi$  {whenever $\varphi\in \CC^N([-1,1])$ for some  $N\in\mathbb N$}.   Then for any $\alpha >\alpha^*_{\fp}(p,q)$  we have 
 \begin{equation}\label{pqr}
  \|\CB^\alpha_*(f,g)\|_{L^r(\R^d)} \lesssim \|f\|_{L^p(\R^d)}\|g\|_{L^q(\R^d)}.
  \end{equation}
 Here $\CC^N([-1,1])$ is a class of smooth functions supported in $[-1,1]$ and  with normalized $C^N$-norm. (See  Section \ref{sec:w} for its precise definition). 
\end{theorem}

Though the statement looks a  bit complicated, the main estimates are those estimates
with $\alpha>\alpha(p)+\alpha(q)+1$ while $p,q,r$ satisfy $\frac1p+\frac1q=\frac1r$  {and $ p,q > \fp\ge \frac{2d}{d-1}$ or  $(p,q)=(2,2)$}. Compared with $L^p\times L^q\to L^r$ boundedness of $\CB^\alpha_1$ in \cite[Theorem 1.1]{JLV}, the lower bound $\alpha(p)+\alpha(q)+1$ on $\alpha$ is exactly 1 larger than that for $\CB^\alpha_1$ even when $p=q=2$.  Roughly speaking, this results from controlling the maximal function by products of  square functions via the Sobolev embedding type inequality which is based on the fundamental theorem of calculus. (see Section \ref{sec:br}). In fact, unlike the linear case we have to apply this argument twice to control the bilinear maximal function. 
  It is not difficult  to see that the constant $C$ in  the estimate \eqref{del-square} depends only on $C^N$-norm of $\varphi$ for some large $N$. 
 Thus,  from currently the best known result regarding the estimates \eqref{del-square} (Theorem \ref{LRS_L})
and Theorem \ref{thm:main} we get the following. 

\begin{corollary}\label{thm:bBR} Let $d\ge 2$ and let $2\le p,q\le \infty$ and $1\le r\le \infty$ satisfy $\frac1p+\frac1q=\frac1r.$ Then the estimate \eqref{pqr} holds whenever $\alpha >\alpha^*_{p_s}(p,q)$, where $p_s=p_s(d)$ in Theorem \ref{LRS_L}. 
\end{corollary}

In particular we note that  $\alpha^*_{p_s}(2,2)=1< \frac{2d+3}4$ for any $d\ge2$. Thus Corollary \ref{thm:bBR}  improves the previously known result due to Grafakos-D.He-Honz\'ik \cite{GHH1} in any dimension. 
Moreover, using a decay of the kernel of $\CB^\alpha_\lambda$, one easily sees that $\CB^\alpha_*$ is bounded from $L^p(\R^d)\times L^q(\R^d)$ to $L^r(\R^d)$ for all $1\le p,q,r\le \infty$ satisfying $\frac1p+\frac1q=\frac1r$ and $\alpha>d-\frac12$.
So, by further interpolation with  these
trivial bounds, we can  improve  the  range of $\alpha$ for which \eqref{pqr} holds.

To show Theorem \ref{thm:main} we mainly rely on the decomposition strategy from 
\cite{JLV} which  reduces the  problem  to dealing with the sublinear operator $f\to \|\FD^\varphi_{\delta,k}f\|_{l^\infty(\mathbb Z)}$, where $\FD^\varphi_{\delta,k}$ is a square function given by \eqref{md:square}.   
Each $\FD^\varphi_{\delta,k}$ contains a linear operator $S^\varphi_{0,\delta,2^k}$ of which multiplier is supported in the  balls of radius $2^k\delta^{1/2}$ which are centered at the origin. So the supports of  multipliers of  $S^\varphi_{0,\delta,2^k}$, $k\in\Z,$ do not overlap boundedly. To get around this lack of orthogonality near the origin
 we consider the operator $f\to \|\FD^\varphi_{\delta,k}f\|_{l^\infty(\mathbb Z)}$ instead of  
 $f\to \|\FD^\varphi_{\delta,k}f\|_{l^2(\mathbb Z)}$. Though the latter is more efficient in capturing cancellation due to orthogonality, the first works better for controlling the maximal function when 
orthogonality between the operators is relatively weak.

The rest of the paper is organized as follows. In Section 2 and Section 3 we provide a  proof of Theorem \ref{thm:bsm} and obtain boundedness of $\wt \CM$.  In Section 4 we reduce the problem of obtaining estimate for  $\CB^\alpha_*$ to  that for an auxiliary operator $\sup_{k\in\Z}\int_1^2 |\FB^\delta_{2^kt}| dt$. 
In Section 5 we introduce the square function $\FD^\varphi_{\delta,k}$ and obtain its maximal bound. By modifying the decomposition lemma in \cite{JLV}, we provide a proof of Proposition \ref{prop:Lp} in Section 6.

Throughout this paper, we use the notation $A\lesssim B$ for positive $A$ and $B$, which means that $A\le CB$ for some $C>0$ independent of $A$ and $B$. Sometimes we write $A\lesssim_\epsilon B$ when the implicit constant depends on  $\epsilon>0$.  We denote by $\CF^{-1}f$ the inverse Fourier transform of $f$, that is to say,
$\CF^{-1}f(x)=\int_{\R^d}e^{2\pi i x\cdot \xi}\widehat f(\xi)d\xi$. For $k\in\mathbb N$, $x\in\R^k,$ and $r>0$, $B^k(x,r)$ denotes the $k$-dimensional ball in $\R^k$ centered at $x$ and of radius $r.$

\section{Proof of Theorem \ref{thm:bsm} } 
In this section we prove $L^p\times L^q\to L^r$ boundedness of the bilinear spherical maximal function $\CM$.  As mentioned before, the boundedness is a direct consequence of a pointwise bound for $\CM$  and the known results regarding the (sub)linear spherical maximal function.  We start by making an  observation concerning pointwise bound for $\CM$.

\begin{lemma}\label{lem:pt}
 Let $d\ge 2$. For any $x\in \R^d$
\begin{equation}\label{pt_bound}
 \CM(f,g)(x)\,\lesssim  \,Mf(x)\, \sS g(x)  \text{ and }   \CM(f,g)(x)\,\lesssim  \sS f(x) Mg(x).
  \end{equation}
 Here $M$ is the Hardy-Littlewood maximal function and $\sS$ is the spherical maximal function.
\end{lemma}
\begin{proof} This pointwise estimate is obtained by a kind of {\it slicing argument} which decomposes the sphere $\mathbb S^{2d-1}$ into a family of lower dimensional spheres.

Let $F$ be a continuous function defined on $\R^{2d}$ and $(x,y)\in\R^d\times \R^d$. We 
first claim that  
\begin{equation}\label{decom}
\begin{aligned}
&\int_{\mathbb S^{2d-1}} F(x,y) d\sigma_{2d-1} (x,y)\\
 =&\int_{B^d(0,1)}\int_{\mathbb S^{d-1}} F(x, \sqrt{1-|x|^2}\,y) (1-|x|^2)^{\frac{d-2}2} d\sigma_{d-1}(y) dx.
\end{aligned}
\end{equation}
Here $d\sigma_{d-1}$ is the induced surface measure on $\mathbb S^{d-1}$. Assuming this for the moment, we proceed to show \eqref{pt_bound}. 
From the equality \eqref{decom}, we see that the bilinear spherical mean is controlled by
\begin{align*}
&\Big|\int_{\mathbb S^{2d-1}} f(x-ty)g(x-tz) d\sigma(y,z)\Big|\\
\le  & \int_{B^d(0,1)} |f(x-ty)| \int_{\mathbb S^{d-1}} |g(x- t\sqrt{1-|y|^2}\,z)|\, d\sigma_{d-1}(z) (1-|y|^2)^{\frac{d-2}2} dy\\
\le & \ \ \sS g(x) \,\int_{B^d(0,1)} |f(x-ty)| \, (1-|y|^2)^{\frac{d-2}2} dy.
\end{align*}
We note that $(1-|y|^2)^{\frac{d-2}2}\le 1$ on $B^d(0,1)$ because $d\ge2.$ Hence we get
\[  \int_{B^d(0,1)} |f(x-ty)|  (1-|y|^2)^{\frac{d-2}2} dy\le  \int_{B^d(0,1)} |f(x-ty)|  dy\lesssim Mf(x),\]
which yields the desired estimate $ \CM(f,g)(x)\,\lesssim  \,Mf(x)\, \sS g(x)$. The other one follows by interchanging the roles of $f$ and $g$. It remains to show \eqref{decom}.

To obtain \eqref{decom}, we make use of the Dirac measure on a hypersurface. Let $\Omega$ be a $(k-1)$-dimensional surface in $\R^k$ given by $\Omega =\{w\in\R^k: \Phi(w)=0\}.$ If $\nabla \Phi(w)\neq 0$ whenever $w\in \Omega$, then it is well-known \cite[p.136]{Ho} that 
\begin{equation}\label{dirac}
\int_{\R^k} G(w) \delta(\Phi) dw =\int_\Omega G(w) \frac{d\nu(w)}{|\nabla \Phi(w)|},
\end{equation}
where $d\nu$ is the induced surface measure on $\Omega$.
Since  $\mathbb S^{2d-1}\subset \R^d\times \R^d$ is the level set $\Phi^{-1}(0)$  of  $\Phi(x,y)=|x|^2+|y|^2-1$, by \eqref{dirac}, we have
\[ \int_{\mathbb S^{2d-1}} F(x,y) d\sigma_{2d-1}(x,y) =2\int_{\R^d\times \R^d} F(x,y)\delta(\Phi) dxdy.\] 
For any $x\in\R^d$, we set $\Phi_x(y)=\Phi(x,y)$ and $\Omega_x =\Phi_x^{-1}(0)\subset \R^d$. Then $\Omega_x$ is empty unless $|x|<1$, and $|\nabla \Phi_x(y)|=2\sqrt{1-|x|^2}\neq 0$ on $\Omega_x$ for $|x|<1$. By applying \eqref{dirac} again with   $\Phi_x(y)$, we obtain
{
\[ 2\int_{\R^d\times \R^d} F(x,y)\delta(\Phi) dxdy =\int_{B^d(0,1)}\int_{\Omega_x} F(x,y) \frac{d\nu_x(y)}{\sqrt{1-|x|^2}} dx,\]
}where $d\nu_x$ is the surface measure on $\Omega_x$. Since $\Omega_x$ is the $(d-1)$-dimensional sphere of radius $\sqrt{1-|x|^2}$, the equality \eqref{decom} follows from scaling.
\end{proof}

  Now we are ready to prove Theorem \ref{thm:bsm}.   

\begin{proof}[Proof of Theorem \ref{thm:bsm}] 
We first prove the necessity part for the estimate \eqref{eq:pqr}. As mentioned above,  the estimate \eqref{eq:pqr} fails unless $p,q$ and $r$ satisfy \eqref{holder} and $r>\frac d{2d-1}$, which follows from  scaling and the counterexample of Barrionevo-Grafakos-D.He-Honz\'ik- Oliveira \cite[Proposition 7]{BGHHO} (see the paragraph below of Theorem \ref{thm:bsm}). Hence, by symmetry, it is enough to show that $\CM$ is not bounded from $L^p\times L^q\to L^r$ with $p=1=r$ and $q=\infty.$  To the contrary, suppose that  $\CM$ was bounded from $L^1(\R^d)\times L^\infty (\R^d)$ to $L^1(\R^d)$.  If we take $g(x)=1$,   by \eqref{decom}  we have,  for any $f\in L^1(\R^d)$, 
\begin{align*}
\CM(f,g)(x) 
&=|\mathbb S^{d-1}|\sup_{t>0} \int_{B(0,1)} |f(x-ty)|  (1-|y|^2)^{\frac{d-2}2} dy\\
&\gtrsim M(f)(x).
\end{align*}
Thus, the assumption for $\CM$ yields
$ \|M(f)\|_{L^1(\R^d)} \lesssim \|f\|_{L^1(\R^d)}$ 
for all $f\in L^1(\R^d)$. This contradicts the fact that  the Hardy-Littlewood maximal function $M$ is not bounded  on $L^1(\R^d)$. 

We now deal with the sufficiency part for the estimate \eqref{eq:pqr}.
To obtain  boundedness of  $\CM$, we shall rely on bilinear interpolation. There is a lot of literature regarding multilinear interpolation  but  it is  usually required for the operator to be linear. See,  \cite{Ja}, \cite{Ch1} and \cite{BOS} for discussion on interpolation in quasi-Banach spaces. However the operator $\CM$ is sublinear. To avoid technicality related to interpolation of  bi-(sub)linear operator we consider  a linearized  operator.  For a nonnegative measurable function $\tau:\mathbb R^d\to \mathbb R_+=
\{t: t\ge 0\}$, define an operator 
\[  \mathcal A_\tau( f,g)(x) =\int_{\bS^{2d-1}}  f(x-\tau(x)y)g(x-\tau(x)z)  d\sigma_{2d-1}(y,z).\] 

To obtain $L^p\times L^q\to L^r$ boundedness of $\CM$, by the Kolmogorov-
Seliverstov-Plessner's stopping time argument it is sufficient to show   
\begin{equation*}\label{eq:pqr-linear}
 \|\mathcal A_\tau(f,g)\|_{L^{r}(\R^d)}\le C\|f\|_{L^{p}(\R^d)}\|g\|_{L^{q}(\R^d)} 
 \end{equation*}  
with $C$ bound independent of $\tau$. Since $|\mathcal A_\tau(f,g)|\le \mathcal M(f,g)$,  by Lemma  \ref{lem:pt}  we see $|\mathcal A_\tau( f,g)(x)|\lesssim Mf(x)\sS g(x)$. It is well known that, for $1<p\le \infty$,
\[ \|Mf \|_{L^p(\R^d)}\lesssim \|f\|_{L^p(\R^d)}, 
\quad  \quad\|Mf \|_{L^{1,\infty}(\R^d)}\lesssim \|f\|_{L^1(\R^d)}.\] 
Also,  we have  $\|\sS f\|_{L^p(\R^d)} \lesssim \|f\|_{L^p(\R^d)}$ if and only if $p>\frac d{d-1}$. Thus, by the pointwise bound for $\mathcal A_\tau$  and  H\"older's inequality,  we obtain for $1< p\le \infty$, $\frac d{d-1}<q\le \infty$ and $r$ satisfying \eqref{holder},
\begin{equation}\label{bsm_main}
 \| \mathcal A_\tau(f,g)\|_{L^r(\R^d)} \le  \|M  f \|_{L^p(\R^d)}\|\sS  g\|_{L^q(\R^d)} \lesssim \|f\|_{L^p(\R^d)}\|g\|_{L^q(\R^d)}.\end{equation}
 Similarly, using the weak $L^1$ bound of the Hardy-Littlewood maximal function, we also have, for $p=1$, $\frac d{d-1}<q\le \infty$ and $r$ satisfying \eqref{holder},
\begin{equation}\label{bsm_weak}
 \| \mathcal A_\tau(f,g)\|_{L^{r,\infty}(\R^d)} \lesssim \|f\|_{L^1(\R^d)}\|g\|_{L^q(\R^d)}. \end{equation}
 Then, interpolation between these estimates \eqref{bsm_weak} yields the estimate \eqref{bsm_main} for $p=1$, $\frac d{d-1}<q<\infty$ and $r$ satisfying \eqref{holder}.
By symmetry, the estimates \eqref{bsm_main} also hold when the roles of $f$ and $g$ are interchanged. Thus,  applying the bilinear (complex) interpolation, we have \eqref{bsm_main} for $p,q,r$ satisfying $\frac1r=\frac1p+\frac1q$ and $(\frac1p,\frac1q)$ in the gray region of Figure \ref{fig}, the closed pentagon $[AGHBO]$ excluding the closed interval $[GH]$ and the points $A$ and $B$.
Note that the implicit constant of \eqref{bsm_main} depends only on $p,q,$ and $d.$ Therefore we obtain all the estimates \eqref{eq:pqr} on  the optimal range.

We next consider the weak  estimates which are included in $(\bf a)$,  $(\bf b)$ and $(\bf c)$.   The estimates in $(\bf c)$  follow from $(\bf a)$, $(\bf b)$, and interpolation. As in the above, to show the maximal bound  \eqref{lorentz}  it is enough to show   
\begin{equation*}\label{eq:pqr-linear}
 \|\mathcal A_\tau(f,g)\|_{L^{r,u}(\R^d)}\le C\|f\|_{L^{p,s}(\R^d)}\|g\|_{L^{q,t}(\R^d)} \end{equation*}  
with a bound $C$ independent of $\tau$.  Thus, $(\bf a)$ follows from the estimate \eqref{bsm_weak} with $q=\infty$ and $r=1$. 
We now show $(\bf b)$ and $(\bf c)$. Let $d\ge 3$. Using the restricted weak type bound due to Bourgain \cite{B0}, 
we get 
\begin{align}
  \|\mathcal A_\tau(f,g) \|_{L^{\frac{d}{2d-1},\infty}}&\lesssim   \|Mf\|_{L^{1,\infty}}\|\sS g\|_{L^{\frac{d}{d-1},\infty}} \lesssim   \|f\|_{1} \|g\|_{L^{\frac d{d-1},1}},  
  \label{1:eq}
  \end{align}
  which shows $(\bf b)$. Interchanging the roles of $f$ and $g$ we also have
  \begin{align*}
\|\mathcal A_\tau(f,g) \|_{L^{\frac{d}{2d-1},\infty}}&\lesssim  \|f\|_{L^{\frac d{d-1},1}} \|g\|_{L^1}.
\end{align*}
Hence, trivially  $\|\mathcal A_\tau(\chi_F,\chi_G) \|_{L^{\frac{d}{2d-1},\infty}}\lesssim  |F|^\frac1p|G|^\frac1q$ for any measurable set
$F$ and $G$ and for $p,$ $q$ satisfying $1\le p,q\le \frac{d}{d-1}$ and $\frac1p+\frac1q=\frac{2d-1}d$. 
Since ${L^{\frac{d}{2d-1},\infty}}$ is $\frac{d}{2d-1}$-convex, the estimates imply 
 $\|\mathcal A_\tau(f,g) \|_{L^{\frac{d}{2d-1},\infty}}\lesssim  \|f\|_{L^{p,  \frac{d}{2d-1}}} \|g\|_{L^{q,\frac{d}{2d-1}}}$ for $p,$ $q$ satisfying   $\frac1p+\frac1q=\frac{2d-1}d$ and  $1\le q\le \frac{d}{d-1}$.
This can be further improved  with bilinear interpolation
to give 
\[  \|\mathcal \mathcal A_\tau(f,g) \|_{L^{\frac{d}{2d-1},\infty}}\lesssim  \|f\|_{L^{p, s}} \|g\|_{L^{q,t}} \]  
for $p,q,s,t$  satsifying $\frac1p+\frac1q=\frac{2d-1}d$ and  $1<q< \frac{d}{d-1}$ and $\frac1s+\frac1t=\frac{2d-1}d$. See Janson \cite{Ja} or Bak-Oberlin-Seeger \cite[Lemma 2.1 and Proposition 2.3]{BOS} for the  bilinear interpolation.
Thus we prove $(\bf c)$. 
\end{proof}

\begin{remark}
\label{strong}
The same  results as in  Theorem \ref{thm:bsm} also hold for the stronger bilinear maximal function $\mathscr M$ which is 
given by
\[\mathscr M(f,g)(x) = \sup_{t,s>0}\int_{\mathbb S^{2d-1}}| f(x-ty)g(x-sz) |d\sigma_{2d-1}(y,z),\quad f,g\in\CS(\R^d). \]
Indeed, we consider a linearized operator
\[\mathcal A_{\tau,\sigma}( f,g)(x) =\int_{\bS^{2d-1}}  f(x-\tau(x)y)g(x-\sigma(x)z)  d\sigma_{2d-1}(y,z)\] with arbitrary measurable functions $\tau$ and $\sigma$.  By the same slicing argument which yields \eqref{decom}, we have 
\[\mathcal A_{\tau,\sigma}(f,g)(x)\,\lesssim  \,Mf(x)\, \sS g(x)  \text{ and }   \mathcal A_{\tau,\sigma}(f,g)(x)\,\lesssim  \sS f(x) Mg(x),\]
hence the previous  argument works  for the operator $\mathscr M$.
\end{remark}

The  results on the spherical maximal function have been extended to maximal averages over  general hypersurfaces  \cite{ SS, sogge, sogge-stein} which vary depending each point.  Naturally, in a similar manner one may consider a bilinear maximal operator $\mathfrak M$  given  by 
\[\mathfrak M(f,g)(x) =\sup_{0<t\le 1}\Big|\int_{\R^{2d}} f(y)g(z) \delta(\Phi_t(x,y,z)) \psi_t(x,y,z) dydz\Big|,\quad f,g\in \CS(\R^d),\]
where $\Phi_t$ and $\psi_t$ are certain smooth functions subject to suitable conditions (for example, see \cite{sogge, sogge-stein}).
It seems to be an interesting problem to characterize the exponents $p,q,r$ for which  $\mathfrak M$ is bounded from $L^p(\R^d)\times L^q(\R^d)$ to $L^r(\R^d)$, however we 
do not attempt to do it in the present paper.

\begin{remark}\label{multi} Our method based on the slicing argument also extends to the general $k$-linear case.
More precisely, for any $k\ge2$,   let   the $k$-(sub)linear spherical maximal function $\CM_k$ be given by
\[ \CM_k(f_1,\cdots, f_k)(x)=\sup_{t>0} \int_{\mathbb S^{kd-1}}\prod_{j=1}^k f_j(x-ty_j)d\sigma_{kd-1}(y_1,\cdots, y_k).\]
Applying our argument inductively, we see that  $\CM_k$ is bounded from $L^{p_1}(\R^d)\times \cdots \times L^{p_k}(\R^d)$ to $L^r(\R^d)$ if  $1<p_1,\cdots, p_k\le \infty$, $\frac1{p_1}+\cdots+ \frac1{p_k}=\frac1r$, $r>\frac d{kd-1}$ and $d\ge2.$
\end{remark}

\section{Localized bilinear spherical maximal function}\label{sec:loc}
In this section we  study  the localized bilinear spherical maximal function $\wt\CM$ defined by \eqref{tu_bsm}. 
Using the $L^p$-improving property for $\wt\sS$, we show $L^p\times L^q\to L^r$ boundedness of $\wt\CM$ for exponents  $p,q,r$ which do not satisfy  the H\"older relation. We  also obtain  necessary conditions on $p,q,r$   for  $L^p\times L^q\to L^r$ boundendess of $\wt\CM$. See  Propositions \ref{pro:mod} and \ref{pro:low1} below.  Consequently,
we obtain the sharp range of $p,q$ while $r$ is restricted in a certain region.


\begin{figure}
\begin{tikzpicture} 
[scale=0.55]
\begin{scope}
	\draw [<->] (0,11.5) node[left]{$\frac1r\,$}--(0,0) node[below]{${\mathbf O}$}--(8.4,0) node[right]{$\frac1p$};
	\path[fill=gray!30] (0,0)--(4,0)--(6, 2.6)--(6,7.8) ;
	\draw (0,0) rectangle (8,10.4);
	\node at (0,10.4)[left] {$2$};
	\draw (0,0)--(6,7.8);
	\draw[dash pattern={ on 2pt off 1pt}] (0,0)--(8,10.4); 
	\draw[dash pattern={ on 2pt off 1pt}] (0,7.8)--(6,7.8); 
	\draw[dash pattern={ on 2pt off 1pt}] (0,5.2)--(8,5.2); 
	\node at (0,5.2)[left] {$1$};  
	\draw[dash pattern={ on 2pt off 1pt}] (0,7.8) node[left]{$\frac32$}--(6,7.8); 
	\draw[dash pattern={ on 2pt off 1pt}] (6,7.8)--(6,0); 
	\draw [fill] (6,2.6) circle [radius=0.05] node[right]{$\mathbf C$};
	\draw[dash pattern={ on 2pt off 1pt}] (0,1.3) node[left]{$\frac14$}--(6,1.3); 
	\draw[dash pattern={ on 2pt off 1pt}] (4,0)--(6,2.6); 
	\draw [fill] (4,0) circle [radius=0.05] node[below]{$\mathbf A$};
	\draw [fill] (6,1.3) circle [radius=0.05] node[right]{$\mathbf B$};
	\draw [fill] (6, 7.8) circle [radius=0.05] node[above left]{$\mathbf D$};
	\draw[pattern=north west lines, pattern color=blue] (4,0) -- (6,1.3)--(6,2.6);
	\draw (4,-1.5) node  {(I) $d=2$};
\end{scope}
\begin{scope}[shift={(12,0)}]
	\draw [<->] (0,11.5) node[left]{$\frac1r\,$}--(0,0) node[below]{${\mathbf O}$}--(8.4,0) node[right]{$\frac1p$};
	\path[fill=gray!30] (0,0)--(4,0)--(6.278,0.562)--(90/13,14/10)--(90/13,9); 
	\node at (0,10.4)[left] {$2$};
	\draw (0,0)--(90/13,9);
	\draw (0,0) rectangle (8,10.4);
	\draw[dash pattern={ on 2pt off 1pt}] (0,0)--(8,10.4); 
	\draw[dash pattern={ on 2pt off 1pt}] (0,5.2)--(8,5.2); 
	\node at (0,5.2)[left] {$1$};  
	\draw[dash pattern={ on 2pt off 1pt}] (0,9) node[left]{$\frac{2d-1}d$}--(90/13,9); 
	\draw[dash pattern={ on 2pt off 1pt}] (90/13,9)--(90/13,0);  
	\draw[dash pattern={ on 2pt off 1pt}] (4,0)--(90/13, 1976/2743); 
	\draw[dash pattern={ on 2pt off 1pt}] (6.278,0.562)--(90/13,14/10); 
	\draw [fill] (4,0) circle [radius=0.05] node[below]{$\mathbf A$};
	\draw [fill] (90/13,1976/2743) circle [radius=0.05] node[right]{$\mathbf B$};
	\draw [fill] (90/13,14/10) circle [radius=0.05] node[above left]{$\mathbf C$};
	\draw [fill] (90/13, 9) circle [radius=0.05] node[above left]{$\mathbf D$};
	\draw [fill] (6.278,0.562) circle [radius=0.05] node[above left]{$\mathbf E$};
	\draw[pattern=north west lines, pattern color=blue] (6.278,0.562)-- (90/13, 1976/2743)--(90/13,14/10);
	\draw(4,-1.5) node {(II) $d\ge 3$};
	\end{scope}
\end{tikzpicture}
\caption{The range of $p$ and $r$ for  $\wt\CM : L^p\times L^p \to L^r$.  Proposition \ref{pro:mod} and  Proposition \ref{pro:low1} give boundedness (the gray region) and unboundedness (the white region), respectively.
Here, ${\mathbf O}=(0,0)$, ${\mathbf A}=(\frac{1}{2},0)$, ${\mathbf B}=(\frac{2d-1}{2d}, \frac{d-1}{d^2})$, $ {\mathbf C}=(\frac{2d-1}{2d},\frac1d)$, $ {\mathbf D}=(\frac{2d-1}{2d},\frac{2d-1}d)$, and ${\mathbf E}=(\frac{2d-3}{2(d-1)}, \frac{d-2}{d(d-1)})$.
}
\label{fig_loc}
\end{figure}
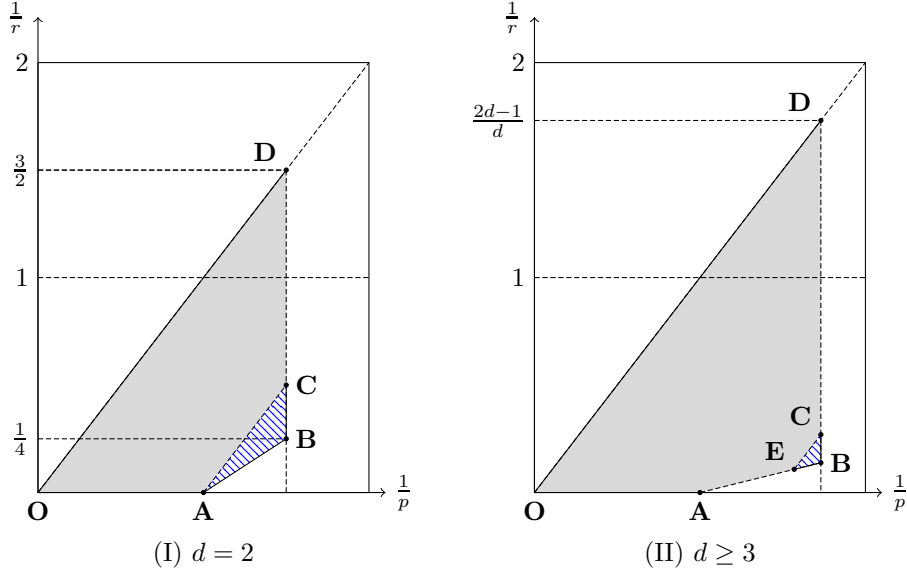
\begin{theorem}\label{thm:local} 
Let $d\ge 2$, $1\le p,q\le \infty $, and $0<r\le d$ or $\frac{d(d-1)}{d-2}\le r<\infty$. Then the estimate 
\begin{equation}\label{loc_pqr} 
\|\wt\CM(f,g)\|_{L^r(\R^d)}\le C \|f\|_{L^p(\R^d)}\|g\|_{L^q(\R^d)}
\end{equation}
holds for $\frac1r\le \frac1p+\frac1q < \min\{\frac{2d-1}d,1+\frac dr \}$. Conversely, the estimate \eqref{loc_pqr} holds only if $\frac1r\le \frac1p+\frac1q \le \min\{\frac{2d-1}d,1+\frac dr \}$.  Furthermore, when $r=\infty$, the estimate \eqref{loc_pqr} holds if and only if $0\le \frac1p+\frac1q\le 1$ for all  $d\ge2.$
\end{theorem}

In particular,  when $d=2$, Theorem \ref{thm:local}  gives \eqref{loc_pqr} for $1\le p,q\le \infty $, and $0<r\le2$ provided that  $ \frac1r\le \frac1p+\frac1q <  \frac 32$.  
Theorem \ref{thm:local} follows from Proposition \ref{pro:mod} and Proposition \ref{pro:low1}.

\begin{prop}\label{pro:mod} Let  $d\ge2$, $1\le p,q\le \infty$ and $0<r\le \infty.$ Then the estimate \eqref{loc_pqr} holds if $r> \frac d{2d-1}$ and $\frac1r\le \frac1p+\frac1q< \min\{ 1+\frac dr, \frac{2d-1}d, \frac1r+\frac{2(d-1)}d\}$ for  $d\ge3$, and if  $r>\frac23$ and $\frac1r\le \frac1p+\frac1q <\min\{ 1+\frac1r, \frac32\}$ for $d=2$. {Moreover, the estimate \eqref{loc_pqr} also holds for $\frac1p+\frac1q=1+\frac1r$ if $r>2$ for $d=2$ and $r=\infty$ for $d\ge3$.}
\end{prop}
\begin{prop}\label{pro:low1} Let $d\ge 2$, $1\le p,q \le \infty$, and $0< r\le \infty.$ If the estimate \eqref{loc_pqr} holds then $\frac1r\le \frac1p+\frac1q\le \min\{ \frac{2d-1}d,\, 1+\frac dr\}.$
\end{prop}

Figure \ref{fig_loc} shows  the range of $p$ and $r$ for which $\wt\CM$ is bounded from $L^p(\R^d)\times L^p(\R^d)$ to $L^r(\R^d)$.   Boundedness of $\wt\CM$ remains open when  $(\frac1p,\frac1r)$ is in the slashed region (the closed triangles with vertices $\bf{A,B,C}$ for $d=2$ and $\bf{E,B,C}$ for $d\ge3$) and the dashed borderlines in Figure \ref{fig_loc}.

We now prove Proposition \ref{pro:mod} and Proposition \ref{pro:low1}. We begin with recalling the known bounds for $\wt\sS$. 
For  $d\ge 2$,  let us set $\mathfrak V_1^d=(0,0)$, $\mathfrak V^d_2=(\frac{d-1}d,\frac{d-1}d),\, \mathfrak V^d_3=(\frac{d-1}d,\frac1d),$ and $\mathfrak V^d_4=(\frac{d^2-d}{d^2+1},\frac{d-1}{d^2+1}).$  By $\Delta(d)$ we denote  the closed quadrangle with vertices $\mathfrak V^d_1,\mathfrak V^d_2, \mathfrak V^d_3,\mathfrak V^d_4$ when $d\ge3$ and the closed triangle with $\mathfrak V^d_1, \mathfrak V^d_2=\mathfrak V^d_3,\mathfrak V^d_4$ when $d=2$.

\begin{theorem}[\cite{Lee}]\label{local_sm} Let $d\ge2$ and $1\le p,q\le \infty$. Then 
\begin{equation}\label{eq:tu_sm}
\|\wt\sS f\|_{L^q(\R^d)}\le C\|f\|_{L^p(\R^d)}
\end{equation}
holds if  $(\frac1p,\frac1q)$ is in $\Delta(d)\setminus\{\mathfrak V^d_2,\mathfrak V^d_3, \mathfrak V^d_4\}$. Conversely, if the estimate \eqref{eq:tu_sm} holds, then $(\frac1p,\frac1q)\in \Delta(d)\setminus\{\mathfrak V^d_2\}$. If $d=2$,  the restricted weak type $(p,q)$ bound for $\wt\sS$ holds with $(\frac1p,\frac1q)=\mathfrak V^d_4$ and if $d\ge 3$, 
with $(\frac1p,\frac1q)=\mathfrak V^d_2,\mathfrak V^d_3,$ and $\mathfrak V^d_4$.   
\end{theorem}

\newcommand{\wa}{\wt{\mathcal A}}
\begin{proof}[Proof of Proposition \ref{pro:mod}] As before, to avoid unnecessary technicality 
we consider a linearized operator. Let $ \kappa:\mathbb R^d\to [1,2]$ be a measurable function and  define an operator   
\[ \wt {\mathcal A}_\kappa ( f,g)(x) =\int_{\bS^{2d-1}}  f(x-\kappa(x)y)g(x-\kappa(x)z)  d\sigma_{2d-1}(y,z).\] 
It is sufficient to show that there is a constant $C$, independent of the measurable function $\kappa$, such that 
\begin{equation}\label{aaa} 
\|\wt {\mathcal A}_\kappa (f,g)\|_{L^r(\R^d)}\le C \|f\|_{L^p(\R^d)}\|g\|_{L^q(\R^d)}
\end{equation}
for $p,q,r$ as in Proposition \ref{pro:mod}. 

Since $\kappa(x)\in [1,2]$, using the same argument (the equality \eqref{decom}) as before, we easily see that
\begin{equation}\label{eq:d1}
\wa_\kappa(f,g)(x)\lesssim |f|\ast \chi_{B}(x) \, \sS g(x)
\end{equation}
and 
\begin{equation}\label{eq:d2}
\wa_\kappa(f,g)(x)\lesssim |f|\ast \chi_{B}(x) \Big( \sum_{l=0}^\infty 2^{-l(\frac{d-2}2)} \wt\sS (g(2^{-l/2}\cdot))(2^{l/2}x) \Big),
\end{equation}
where $B=B^d(0,2)$, the $d$-dimensional ball of radius $2$. Indeed, the estimate \eqref{eq:d1} is a direct consequence of the equality \eqref{decom}. To show \eqref{eq:d2}, we dyadically decompose the ball $B^d(0,1)$  
away from its boundary. 
For $l\ge 1$, let us set  $\mathbb A_l=\{y\in\R^d \,: 1-2^{-l-1}\le |y|\le 1-2^{-l-2} \},$ and
\begin{align*}
\CI_l &=\int_{\mathbb A_l} |f(x-\kappa(x)y)| (1-|y|^2)^{\frac{d-2}2} \int_{\mathbb S^{d-1}} |g(x- \kappa(x)\sqrt{1-|y|^2}\,z)|\, d\sigma_{d-1}(z) \,dy,\\
\CI_0 &=\int_{B^d(0,\frac34)} |f(x-\kappa(x)y)| (1-|y|^2)^{\frac{d-2}2} \int_{\mathbb S^{d-1}} |g(x- \kappa(x)\sqrt{1-|y|^2}\,z)|\, d\sigma_{d-1}(z) \,dy.
\end{align*}
Then by \eqref{decom} it follows that 
\[ \wa_\kappa(f,g)(x) \le \sum_{l=0}^\infty \CI_l.
\]
Note that, for $y\in\mathbb A_l$,  $1-|y|^2 =(1-|y|)(1+|y|) \sim 2^{-l}$ and $\kappa(x)\mathbb A_l =\{\kappa(x)y : y\in \mathbb A_l\}$ is included in $B^d(0,2)$ because $1\le \kappa(x)\le 2$. Hence, by scaling  we have
\begin{align*}
 \CI_l &\lesssim 2^{-l(\frac{d-2}2)} |f|\ast \chi_B(x) \wt\sS (g(2^{-l/2}\cdot))(2^{l/2}x),\, l\ge 1.
\end{align*}
Since $1-|y|\sim 1$  for $y\in B^d(0,3/4)$,  we  have $\CI_0 \lesssim |f|\ast \chi_B(x) \wt\sS g(x)$. Thus we get  \eqref{eq:d2}.


We first use \eqref{eq:d1} to obtain  the estimate  \eqref{aaa}. Let $d\ge2$ and $r>\frac{d}{2d-1}$. By Young's convolution inequality, it is clear that $\||f|\ast \chi_{B^d(0,2)}\|_{L^{p_2}(\R^d)}\lesssim \|f\|_{L^{p_1}(\R^d)}$ for any $1\le p_1\le p_2\le \infty$. Combining this with the known bounds for $\sS$, we have 
\begin{equation}\label{tem_lo}
\|\wa_\kappa(f,g)\|_{L^r(\R^d)}\le \| |f|\ast \chi_B\|_{L^u(\R^d)}\|\sS g\|_{L^q(\R^d)}\lesssim \| f\|_{L^p(\R^d)}\|  g\|_{L^q(\R^d)}
\end{equation}
for  $\frac1r=\frac1u+\frac1q$, $\frac1q< \min\{\frac{d-1}d,\frac1r\}$, and $\frac1u\le \frac1p\le 1.$ Moreover, if $r>\frac{d}{d-1}$, then the estimate \eqref{tem_lo} also holds for $q=r$ and $0\le \frac1p\le 1$ because of the previous estimate for $\sS$. By symmetry we may interchange the roles of $f$ and $g$. So,  we have \eqref{aaa} for $\frac1r-\frac1p\le \frac1q \le 1$ and   $\frac1p \le \frac1r$ if $r>\frac d{d-1}$, and for $\frac1r-\frac1p\le \frac1q \le 1$ and $\frac1p< \frac{d-1}d$ if $r\le \frac d{d-1}$. Thus, from (complex) interpolation
we obtain the estimate \eqref{aaa} whenever
$(\frac1p,\frac1q,\frac1r)$ is in $\mathfrak T(r)\times \{\frac1r\}$, 
 where
$\mathfrak T(r)$ is given by
\[
\mathfrak T(r):=
\begin{cases}
\{(\frac1p,\frac1q)\in[0,1]^2:  \frac1r \le \frac1p+\frac1q \le  1+\frac1r \}, &\text{ if }  r>\frac{d}{d-1}, 
\\[5pt]
\{(\frac1p,\frac1q)\in[0,1]^2:  \frac1r \le \frac1p+\frac1q< \frac{2d-1}{d} \} ,   & \text{ if }  r\le \frac{d}{d-1}.
\end{cases}
\]
 Therefore we obtain the desired estimates for $d=2$.  

We  turn to the case $d\ge3$. By using the inequality \eqref{eq:d2}, we can further  extend the range of $p,q$ for \eqref{aaa},   when $r<\frac {d-1}d$. We let  $r<\frac{d-1}d$. If $(\frac1q,\frac1r)\in \Delta(d)\setminus\{\mathfrak V^d_2,\mathfrak V^d_3,\mathfrak V^d_4\}$,  Theorem \ref{local_sm} and scaling imply
\[ \|\wt\sS (g(2^{-l/2}\cdot))(2^{l/2}x)\|_{L^r(\R^d)} \lesssim2^{\frac{ld}2(\frac1q-\frac1r)}\|g\|_{L^q(\R^d)}.\]
Using this and  \eqref{eq:d2},  we see for any $1\le p\le \infty$, 
\begin{equation*}\label{border}
\begin{aligned}
\|\wa_\kappa(f,g)\|_{L^r(\R^d)}&\lesssim \| |f|\ast \chi_{B}\|_{L^\infty(\R^d)}  \Big\| \sum_{l=0}^\infty 2^{-l(\frac{d-2}2)} \wt\sS (g(2^{-l/2}\cdot))(2^{l/2}x) \Big\|_{L^r(\R^d)}\\
&\lesssim  \Big(\sum_{l=0}^\infty 2^{-\frac l2 (d-2 -d(\frac1q-\frac1r))}\Big) \| f\|_{L^p(\R^d)}\| g\|_{L^q(\R^d)}.
\end{aligned}
\end{equation*}
Thus, we have \eqref{aaa} whenever $1\le p\le \infty$, $\frac 1q -\frac1r <\frac{d-2}d$,  and $(\frac1q,\frac1r)\in \Delta(d)\setminus\{ \mathfrak V^d_2, \mathfrak V^d_3,\mathfrak V^d_4\}$.

For $(\frac1q,\frac1r)$ in $\Delta(d)\setminus\{\mathfrak V^d_2,\mathfrak V^d_3,\mathfrak V^d_4\}$  we separately consider the following  three cases: 
\[\mathbf A: \frac1d<\frac1r<\frac{d-1}d,  \    \  \mathbf B:\frac1r<\frac{d-2}{d(d-1)}, \  \   \mathbf C:\frac{d-2}{d(d-1)}\le \frac1r\le \frac 1d.\]
 In the case $\mathbf A$,  $(\frac1q,\frac1r)\in \Delta(d)\setminus\{\mathfrak V^d_2,\mathfrak V^d_3,\mathfrak V^d_4\}$ if and only if $\frac1r \le \frac1q<\frac {d-1}d$, so $\frac1q-\frac1r<\frac{d-2}d$ holds. Thus we have \eqref{aaa} for  $1\le p\le \infty$ and  $\frac1r \le \frac1q<\frac {d-1}d$. Notice that the estimate \eqref{aaa} is also true for $\frac1r-\frac1q\le \frac1p\le1$ and $\frac1q\le \frac1r$ (see \eqref{tem_lo}).
 From this and symmetry, we see that the estimate \eqref{aaa} holds for $\max\{\frac1r-\frac1q,0\}\le \frac1p\le1$ and $\frac1q<\frac{d-1}d$, or $\max\{\frac1r-\frac1p,0\}\le \frac1q\le1$ and $\frac1p<\frac{d-1}d$. Then interpolation between these estimates gives the estimate \eqref{aaa} for  $\frac1d<\frac1r<\frac{d-1}d$
and  $(\frac1p,\frac1q)\in \mathfrak D(r)$ which is given by 
\[\mathfrak D(r):=\Big\{(\frac1p,\frac1q)\in [0,1]^2 :  \frac1r\le\frac1p+\frac1q<\frac{2d-1}d\Big\}, \ \   \frac{d}{d-1}<r<d.\]
 
 In the case $\mathbf B$, $(\frac1q,\frac1r)\in \Delta(d)\setminus\{\mathfrak V^d_2,\mathfrak V^d_3,\mathfrak V^d_4\}$ if and only if $\frac1r\le \frac1q \le \frac dr$, so $\frac1q-\frac1r<\frac{d-2}d$. Thus we have \eqref{aaa} for  $1\le p\le \infty$ and  $\frac1r\le \frac1q \le \frac dr$. In the case  $\mathbf C$, $(\frac1q,\frac1r)$ is in  $ \Delta(d)\setminus\{\mathfrak V^d_2,\mathfrak V^d_3,\mathfrak V^d_4\}$ whenever $\frac1q-\frac1r<\frac{d-2}d$. (Note that  there is $q_\circ$ such that $(\frac1{q_\circ},\frac1r)\in \Delta(d)$ but $\frac1{q_\circ}-\frac1r\ge \frac{d-2}d$). Hence, we have  \eqref{aaa} for $1\le p\le \infty$
 and $\frac 1r\le \frac1q<\frac1r+\frac{d-2}d$. 
We now define the set  $\mathfrak D(r)$  for the cases $\mathbf B$ and $\mathbf C$ by 
  \[
 \mathfrak D(r):=
\begin{cases}
\{(\frac1p,\frac1q)\in[0,1]^2 :  \frac1r\le \frac1p+\frac1q<\frac dr+1\}, &\text{ if }  \frac1r<\frac{d-2}{d(d-1)}, 
\\[5pt]
\{(\frac1p,\frac1q)\in[0,1]^2 :  \frac1r\le\frac1p+\frac1q<\frac 1r+\frac{2d-2}d\} ,   & \text{ if }  \frac{d-2}{d(d-1)}\le \frac1r\le \frac1d.
\end{cases}
\]
 Then by symmetry and applying interpolation again, we have  \eqref{aaa} for $(\frac1p,\frac1q)\in \mathfrak D(r)$, $0\le \frac1r\le \frac1d$.  As a result, we see that the estimate \eqref{aaa} holds for $p,q,r$ whenever $(\frac1p,\frac1q)$ is in $\mathfrak D(r) \cup \mathfrak T(r)$ and $r>\frac d{2d-1}$, which completes the proof.
 \end{proof}

\begin{proof}[Proof of Proposition \ref{pro:low1}] 
It is easy to see that the estimate \eqref{loc_pqr} is impossible when $\frac1p+\frac1q <\frac1r$, since the bilinear spherical mean is commutative with simultaneous translation \cite{BeGe}.  So it is enough to show that the exponents $p,q,r$ should satisfy 
$\frac1p+\frac1q\le \min\{ \frac{2d-1}d,\, 1+\frac dr\}$ 
whenever we assume that the estimate \eqref{loc_pqr} holds.

Now we assume \eqref{loc_pqr}. We first show that $\frac1p+\frac1q\le \frac{2d-1}d.$ Let us fix $0<\epsilon_\circ\ll 1$ sufficiently small. For $0<\delta\le \epsilon_\circ$, we set 
\[ f_\delta=\chi_{B^d(0,\delta)}\quad \mbox{and}\quad g_\delta =\chi_{B^d(0,C_1\delta)},\]
where $C_1$ is a constant  chosen later. Let $\mathbb A:=\{x\in\R^d : \frac1{\sqrt 2}\le |x|\le \frac1{\sqrt 2} +\epsilon_\circ\}.$ Then we claim  that, for any $x\in\mathbb A$ and $0<\delta\le\epsilon_\circ$,
\begin{equation}\label{lower1}
 \wt \CM(f_\delta,g_\delta)(x)\ge C \delta^{2d-1}
 \end{equation}
with $C>0$  independent of $\delta.$  So, the estimate \eqref{loc_pqr} implies
\[ |\mathbb A|^{\frac1r}\delta^{2d-1}\lesssim  \delta^{\frac dp +\frac dq},\quad 0<\delta \ll 1, \]
which yields  $\frac1p+\frac1q\le \frac{2d-1}d$ by letting  $\delta\to 0$. 

We now show \eqref{lower1}. To do so, for $x\in \mathbb A$ we set
\begin{align*}
E^1_x &=\{y\in\R^d :  |x / |x| - \sqrt 2y |\le \sqrt 2\delta/(1+\sqrt{2}\epsilon_\circ) \}\quad \mbox{and}\\
E^2_x &=\{z\in\R^d :  |z|=1,\,\, | x /{ |x|} - z |\le \sqrt 2 C_2\delta/(1+\sqrt{2}\epsilon_\circ) \},
\end{align*}
where $C_2$ is a constant which is chosen later. Then for $x\in\mathbb A$ and $y\in E^1_x$ it is easy to check that $f_\delta(x-\sqrt{2}|x| y )=1$ and $|\sqrt{1-|y|^2}-1/\sqrt{2}|\le C_3 \delta$ for some $C_3>0$ depending only on $\epsilon_\circ.$ We now put $C_2=C_3.$ Then for  $z\in E^2_x$ $|x-\sqrt{2}|x|\sqrt{1-|y|^2}z|\le 3C_2\delta $, hence $g_\delta(x-\sqrt{2}|x|\sqrt{1-|y|^2}z)=1$ when we choose $C_1$ so that $C_1>3C_2$. From this and \eqref{decom}, we see that for $x\in \mathbb A$
\begin{align*}
\wt \CM(f_\delta,g_\delta)(x) 
&\ge \Big|\int_{\mathbb S^{2d-1}} f_\delta(x-\sqrt{2}|x|y)g_\delta(x-\sqrt{2}|x|z) d\sigma_{2d-1}(y,z)\Big|\\
&\ge \int_{E^1_x} f_\delta(x-\sqrt{2}|x|y)\int_{\mathbb S^{d-1}\cap E^2_x} 
g_\delta(x-\sqrt{2}|x|z) d\sigma(z)(1-|y|^2)^{\frac{d-2}2} dy\\
&\ge (1/\sqrt{2} -C_2\delta)^{d-2} |E^1_x| \sigma(E^2_x) \sim \delta^{d}\delta^{d-1},\quad 0<\delta\ll 1.
\end{align*}
Here the implicit constant only depends on $\epsilon_\circ$, hence we obtain \eqref{lower1}.

We next show that $\frac1p+\frac1q\le 1+\frac dr.$  We fix small $\epsilon>0$ and set for  $0<\delta\le \epsilon$  
\[f_\delta=\chi_{B(0,\frac1{\sqrt{2}}+2\delta)\setminus B(0,\frac1{\sqrt{2}}-2\delta)}\quad  \mbox{and} \quad g_\delta=\chi_{B(0,\frac1{\sqrt{2}}+C_1\delta)\setminus B(0,\frac1{\sqrt{2}}-C_1\delta)}\] 
for some $1<C_1 < 2^{-3/2}\epsilon^{-1}$ which is to be chosen later. Then, if $|x|\le \delta$ and $\frac1{\sqrt{2}}-\delta\le |y|\le \frac1{\sqrt 2}$ we have  $f_\delta(x-y)=1$ and $\frac1{\sqrt 2} \le \sqrt{1-|y|^2} \le  \frac1{\sqrt{2}}+C_2\delta$ for some $C_2$ depending only on $\epsilon$. Hence, $g_\delta(x-\sqrt{1-|y|^2}z)=1$ for $|z|=1$, $|x|\le \delta$, and $\frac1{\sqrt{2}}-\delta\le |y|\le \frac1{\sqrt 2}$,  if we choose $C_1$ so that $C_1>C_2+1$. Thus, by the equality \eqref{decom} we have for $|x|\le \delta$
\begin{align*}
\wt{\CM}(f_\delta,g_\delta)(x)
&\ge \int_{\mathbb S^{2d-1}} f_\delta(x-y)g_\delta(x-z) d\sigma_{2d-1}(y,z)
 \\
&\gtrsim \int_{\{y\,: \,2^{-1/2}-\delta\le |y|\le 2^{-1/2}\}} f_\delta(x-y)\int_{\mathbb S^{d-1}} 
g_\delta(x-\sqrt{1-|y|^2}z) d\sigma(z) dy\\
&\gtrsim \delta.
\end{align*}
The $L^p\times L^q \to L^r$ boundedness of $\wt \CM$ implies $\delta^{1+\frac dr} \lesssim \delta^{\frac1p+\frac1q},$ $ 0<\delta\ll 1$, hence we get the desired $\frac1p+\frac1q\le 1+\frac dr.$ 
\end{proof}

\section{Proof of Theorem \ref{thm:main}}\label{sec:br}
In this section, we reduce  the maximal estimate for the bilinear Bochner-Riesz operator  to that for a maximal operator generated by  bilinear multiplier operators of which multipliers supported in a thin annulus. To do this, we break $\CB^\alpha_\lambda$ into a sum of auxiliary operators $\FB^\delta_\lambda$, $0<\delta\le 1/4$, by decomposing the multiplier of $\CB^\alpha_\lambda$ dyadically away from its  singularity $\{(\xi,\eta) : |\xi|^2+|\eta|^2=1/\lambda^2\}$.

More precisely,  let us choose  $\psi\in C^\infty_0([\frac12,2])$ and $\psi_0\in C^\infty_0([-\frac34,\frac34]) $ such that
$(1-t)^\alpha_+ =\sum_{j=2}^\infty 2^{-j\alpha} \psi (2^j(1-t)) +\psi_0(t),$ $0\le t<1$. Using this, we have
\begin{equation}\label{reduction} \CB^\alpha_*(f_1,f_2)(x)\le  \sum_{\delta \in \mathcal{D}} \delta^{\alpha} \sup_{\lambda>0}|\FB^\delta_\lambda(f_1,f_2)(x)| +\sup_{\lambda>0}|T_{m_\lambda}(f_1,f_2)(x)|,
\end{equation}
where $\mathcal{D}$ is the set of positive dyadic numbers $\le 1/4$,
\[\FB^\delta_\lambda(f,g)(x)=\iint_{\R^d\times\R^d} e^{2\pi i x\cdot(\xi+\eta)}\psi\Big(\frac{1-|\lambda\xi|^2-|\lambda\eta|^2}\delta\Big)\wh{f}(\xi)\wh g(\eta)d\xi d\eta, 
\] 
and $T_{m_\lambda}$ is a bilinear multiplier operator with multiplier $m_\lambda =m(\lambda\cdot)$ and $m(\xi,\eta) =\psi_0(|\xi|^2+|\eta|^2)$.  Since $m$ is smooth and supported in a compact set, it is easy to see that $\sup_{\lambda>0}|T_{m_\lambda}(f_1,f_2)(x)|$ is  dominated by the product of the Hardy-Littlewood maximal functions of  $f$ and $g$. Hence  we have, for $1<p,q \le \infty$ and $r$ satisfying $\frac1r=\frac1p+\frac1q$,
\[ \Big\|\sup_{\lambda>0}|T_{m_\lambda}(f_1,f_2)|\Big\|_{L^{r}(\R^d)} \lesssim \|f_1\|_{L^p(\R^d)}\|f_2\|_{L^q(\R^d)}.\]
Thus, the major task is to get bound on the maximal function $ \sup_{\lambda>0}|\FB^\delta_\lambda(f_1,f_2)|$ in terms of $\delta$. From now on we focus on obtaining estimates for the maximal operator 
\[\FB^\delta_* (f_1,f_2):=\sup_{\lambda>0}|\FB^\delta_\lambda(f_1,f_2)|, \  \   0<\delta\le 1/4.\]

To deal with $\FB^\delta_*$, we adopt the standard arguments relating the maximal operator to the square function (\cite{CaS,Stein2}). Especially,  by the fundamental theorem of calculus, 
 $|F(t)| \le |F(s)| +\int_1^2|F'(\tau)|d\tau$, $1\le s,t\le 2$. Hence we obtain
\begin{equation*}\label{re1}
\begin{aligned}
&\FB^\delta_*(f_1,f_2)(x)
= \sup_{k\in \Z} \sup_{1\le \lambda\le 2 } \big|\FB^\delta_{2^k\lambda}(f_1,f_2)(x)\big|\\
&\quad \le \sup_{k\in \Z} \Big(\int_1^2 \big |\FB^\delta_{2^k t}(f_1,f_2)(x) \big | dt +\int_1^2 \big |\frac{\partial}{\partial t}\FB^\delta_{2^k t}(f_1,f_2)(x) \big | dt\Big).
 \end{aligned}
 \end{equation*}
Notice that $\delta \frac{\partial}{\partial t}\FB^\delta_{2^k t}$ satisfies the same quantitative properties as $\FB^\delta_{2^k t}$ when $1\le t\le 2$, since $\frac{\partial}{\partial t}\FB^\delta_{2^k t}$ is also a bilinear multiplier operator with a multiplier $m_t(2^kt\xi,2^kt\eta)$ where $ m_t(\xi,\eta) =\frac{-2}{t \delta}(|\xi|^2+|\eta|^2) \psi'(\frac1\delta(1-|\xi|^2-|\eta|^2)).$ Hence in order to estimate $\FB^\delta_*$ it suffices to consider  
the operator
\begin{equation*}\label{eq:max}  
(f_1,f_2)\, \to \,\sup_{k\in \Z}\int_1^2 \big|\FB^\delta_{2^k\lambda } (f_1,f_2)(x)\big|d\lambda.
\end{equation*}
In fact, Theorem \ref{thm:main} is an immediate consequence of Proposition \ref{prop:Lp} below.
\begin{prop}\label{prop:Lp}
Let $d\ge 2$ and $2\le p,q<\infty$ and $1\le r\le \infty$ satisfy $\frac1p+\frac1q=\frac1r.$ Set $\mathfrak p>\frac{2d}{d-1}$. Suppose that for $s>\mathfrak p$
the estimate \eqref{del-square} holds uniformly in $\phi\in \CC^{N_\circ}([-1,1])$ for some $N_\circ\in\mathbb N.$ Then we have, for $0<\delta\le 1/4$ and $\epsilon>0$,  
\begin{equation}\label{eq:Lp}
\Big\|  \sup_{k\in\Z}\int_1^2\big|\FB^\delta_{2^k\lambda}(f_1,f_2)(x)\big|d\lambda \Big\|_{L^{r}(\R^d)}\le C_\epsilon \delta^{-\alpha_{\fp}(p,q)+1-\epsilon} \|f_1\|_{L^p(\R^d)}\|f_2\|_{L^q(\R^d)}.
\end{equation}
Here $\CC^{N_\circ}([-1,1])$ is a class of smooth functions defined in Section \ref{sec:w} and $\alpha_{\fp}$ is given by \eqref{alpha}.
\end{prop}

From Proposition \ref{prop:Lp} (and Theorem \ref{LRS_L}), we particularly see that,   for any $\epsilon>0$,
\begin{equation*}\label{L2}
\Big\|  \sup_{k\in\Z}\int_1^2\big|\FB^\delta_{2^k\lambda}(f_1,f_2)(x)\big|d\lambda \Big\|_{L^{1}(\R^d)}\le C_\epsilon \delta^{-\epsilon} \|f_1\|_{L^2(\R^d)}\|f_2\|_{L^2(\R^d)}.
\end{equation*}
Compared with the linear case, there is no gain of $\delta$-exponent in this step. More precisely,  its linear counterpart is the estimate
\begin{equation} \label{LL2}
\left\|\sup_{k\in\Z}\Big(\int_1^2|\psi\Big(\frac{1-|2^ktD|^2}\delta\Big)f|^2dt\Big)^{1/2}\right\|_{L^2(\R^d)}\lesssim \delta^{1/2} \|f\|_{L^2(\R^d)},
\end{equation}
which follows from Plancherel's theorem and the Littlewood-Paley inequality. The positive power of  $\delta$  in  \eqref{LL2} offsets the negative power of $\delta$ which occurs in the standard argument relating the maximal estimate to the  square function estimate. So, one can prove $L^2$-boundedness of   $\CR^\alpha_*$ for $\alpha>0.$ However, this is not the case with  the  bilinear maximal operator, hence we only obtain $L^2\times L^2\to L^1 $ boundedness of $\CB^\alpha_*$ for $\alpha>1$.

Before finishing this section, we make a remark on the  negative results  regarding $L^p\times L^q\to L^r$ boundedness of $\CB^\alpha_*.$
 A necessary condition for $L^p\times L^q \to L^r$ boundedness of $\CB^\alpha_*$ was obtained by D. He in \cite{He}.  He showed that $\CB^\alpha_*$ is unbounded from $L^p(\R^d)\times L^q(\R^d)$ to $L^r(\R^d)$ if $\alpha < \frac{2d-1}{2r}-\frac{2d-1}2$, by adopting the counterexample for the maximal Bochner-Riesz operator in \cite{tao}. In particular this shows that  $L^2\times L^2\to L^1$ boundedness  holds only if $\alpha\ge 0$.

%

\section{$L^p$-estimate for a mixed square function }\label{sec:w}
In this section we obtain $L^p(l^\infty)$-estimates for a square function $\FD^\varphi_{\delta, k}$ defined by \eqref{md:square},  which plays a key role in the proof of Propositions \ref{prop:Lp}. 

To define $\FD^\varphi_{\delta, k}$, let $I=[-1,1]$ and consider the class of smooth functions \[\CC^N(I):=\Big\{ \varphi: \supp\, \varphi\subset I, \|\varphi\|_{C^N(\R)}:=\max_{0\le n\le N }\|\frac{d^n}{dt^n}\varphi\|_{L^\infty(\R)}\le 1\Big\}.\]
For $\varphi\in\CC^N(I)$ and positive numbers $\rho,\delta,\lambda>0$,  we define a (linear) multiplier operator $S^\varphi_{\rho,\delta, \lambda}$ 
 by setting
\begin{equation*}\label{opS}
\CF^{-1}\big(S^\varphi_{\rho,\delta, \lambda} f \big)(\xi) =\varphi\Big(\frac{\rho-|\lambda\xi|^2}{\delta}\Big)\wh f(\xi),\quad f\in \CS(\R^d),
\end{equation*}
and  a mixed square function $\FD^\varphi_{\delta, k}$, $k\in\Z$,  by
\begin{equation}\label{md:square}
\FD^\varphi_{\delta,k}f(x) = \Big(\sum_{\rho\in\delta\Z\cap [0,2]} \int_1^2 |S^\varphi_{\rho,\delta, 2^k\lambda}f(x)|^2 d\lambda\Big)^{1/2}, \quad f\in \CS(\R^d).
\end{equation}
From now on, we write $S^\varphi_{\rho,\delta} =S^\varphi_{\rho,\delta,1}$ for simplicity. 

\begin{prop}\label{prop:square} Let $d\ge2$, $2\le p< \infty$, and $N\ge d$.  Suppose that, for  any $0<\delta \le 1/4$,
\begin{equation}\label{tem:sq}
\Big\| \Big(\int_{1/2}^1 |S^\varphi_{t,\delta}f(x)|^2 dt\Big)^{1/2}\Big\|_{L^p(\R^d)} \le C\delta^{-\beta}\|f\|_{L^p(\R^d)}
\end{equation}
holds with some $\beta\ge -\frac12$, and $C$ independent of $\varphi\in\CC^N(I)$. Then for any $\epsilon>0$ we have 
\begin{equation}\label{eq:square}
\big\| \sup_{k\in \Z} | \FD^\varphi_{\delta,k}f(x)|\big\|_{L^p(\R^d)}  \lesssim_\epsilon \delta^{-\beta-\frac12-\epsilon} \|f\|_{L^p(\R^d)}
\end{equation}
with $C$ independent of  $\varphi\in \CC^{N+1}(I)$ and $0<\delta\le1/4$.
\end{prop}

Using  Plancherel's theorem,  it is easy to check that   the estimate \eqref{tem:sq} holds  with $p=2$, $\beta=-\frac12$, and  a uniform $C$ as long as $\varphi\in\CC^N(I)$ for any $N\in\mathbb N$.  This is essentially due to Stein \cite{stein0}. Hence Corollary \ref{lem:L2} below is a direct consequence of Proposition \ref{prop:square}.
  
\begin{corollary}\label{lem:L2} Let $d\ge2$, $0\le \delta\le 1/4$, and $N\ge d+1.$ Then, for any $\epsilon>0$,  there is a constant $C=C(\epsilon)$ such that 
\[ \big\|  \sup_{k\in \Z} |\FD^\varphi_{\delta,k} f |\big \|_{L^2(\R^d)}\lesssim \delta^{-\epsilon} \|f\|_{L^2(\R^d)}\]
holds with $C$ uniform as long as $\varphi\in \CC^N(I)$. 
\end{corollary}

Before proving Proposition \ref{prop:square}, we recall the square function $\mathfrak S^\varphi_\delta$ associated with the Bochner-Riesz operator, which is given by \eqref{lo-square}. As mentioned in the introduction, the sharp $L^p$-estimate for $\mathfrak S^\varphi_\delta$ (the estimate \eqref{del-square}) has been studied by various authors. Among them  currently the best known results were obtained by Lee-Rogers-Seeger \cite{LRS} and Lee \cite{lee2}. We summarize them in Theorem \ref{LRS_L} below. 

\begin{theorem} \cite{LRS,lee2}\label{LRS_L} 
Let $d\ge2$. Then the estimate \eqref{del-square} holds if $p>p_s(d)=\min\{ p_0(d),\frac{2(d+2)}d\}$, where 
 $p_0(d) =2+\frac{12}{4d-6-k},~d\equiv k ~(\mbox{mod } 3),~k=0,1,2.$
\end{theorem}

Since the square function presented in the estimate \eqref{tem:sq} is bounded by   $\mathfrak S^\varphi_\delta f(x)$, by Theorem \ref{LRS_L}  and Proposition \ref{prop:square} we can obtaine $L^p (l^\infty)$-estimate for the mixed square function $\mathfrak D^\varphi_{\delta,k}$  for $p\ge 2.$ 

We now prove Proposition \ref{prop:square}.

\begin{proof}[Proof of Proposition \ref{prop:square}]
To estimate $\FD^\varphi_{\delta,k}f,$ we
 first decompose the interval $[0,2]$ into dyadic subintervals as follows:
\[  [0,2]=[0,4\delta]\cup[4\delta,2], \  \  \ [4\delta,2]=\bigcup_{j=-1}^{j_\circ} \BI_j:= \bigcup_{j=-1}^{j_\circ} [4\delta, 2]\cap [2^{-j-1},2^{-j}],\]
where $j_\circ$ is the smallest integer satisfying $4\delta\ge 2^{-j_\circ-1}.$
Then by the triangle inequality, the left-hand side of \eqref{eq:square} is bounded by $\sum_{j=-1}^{j_\circ}\CI_j +\CI\CI$, where
\begin{align*}
(\CI_j )^p &=\int_{\R^d} \Big(\sup_{k\in \Z}\int_1^2 \sum_{\rho\in\delta\Z\cap \BI_j} |S^\varphi_{\rho,\delta, 2^k\lambda}f(x)|^2 d\lambda \Big)^{p/2} dx,\quad -1\le j\le j_\circ,\\
(\CI\CI)^p&= \int_{\R^d} \Big(\sup_{k\in \Z}\int_1^2 \sum_{\rho\in\delta\Z\cap [0,4\delta]} |S^\varphi_{\rho,\delta, 2^k\lambda}f(x)|^2 d\lambda \Big)^{p/2} dx.
\end{align*}
Note that each $S^\varphi_{\rho,\delta,\lambda}$ satisfies, for $t>0$,
\begin{equation}\label{scaling}
S^\varphi_{\rho,\delta}f(x) = S^\varphi_{t\rho, t\delta} \,f_{t^{1/2}}(t^{-1/2}x), \ \mbox{  and  } \ S^\varphi_{\rho,\delta,t}f(x) = S^\varphi_{\rho, \delta} \,f_{t}(t^{-1}x),
\end{equation}
where $f_t =f(t\,\cdot)$.
By these relations and scaling, in order to prove Proposition \ref{prop:square}, it suffices to deal with  $\CI_0$ and $\CI\CI$. 
 More precisely, we will show that  for $0<\delta\le1/4$ 
\begin{align}
\max\left\{\,\CI_0, \,\CI\CI\,\right\}&\lesssim \delta^{-\beta-\frac12}\|f\|_{L^p(\R^d)}\label{I0}
\end{align}
with   uniform implicit constant  as long as  $\varphi \in \CC^{N+1}(I)$. 
Indeed, by  the first relation in \eqref{scaling}, for $-1\le j\le j_\circ$ we see that
\[
\CI_j \le  2^{jd/2p} \Big\| \Big(\sup_{k\in \Z}\int_1^2 \sum_{\rho\in 2^{j}\delta\Z\cap [1/2,1]} |S^\varphi_{\rho,2^j\delta, 2^k\lambda}f_{2^{j/2}}|^2 \Big)^{1/2}\Big\|_{L^p(\R^d)}.\]
Since $2^{-j}\ge 2^{-j_\circ}>4\delta$, $2^j\delta <1/4$ for any $0<\delta\le 1/4$. Thus we can apply \eqref{I0} to obtain
$\CI_j\lesssim (2^j \delta)^{-\beta-\frac12} \|f\|_{L^p(\R^d)}.$
Here the implicit constant is independent of the choice of $\varphi\in \CC^{N+1}(I)$, $\delta$, and $j$. Note that $-\beta-\frac12\le 0$.
Taking the summation over $j$, we have, for any $\epsilon>0$,
\[\| \sup_{k\in \Z} | \FD^\varphi_{\delta,k}f|\|_{L^p(\R^d)}
\lesssim j_\circ\delta^{-\beta-\frac12} \|f\|_{L^p(\R^d)}\le  C_\epsilon \delta^{-\beta-\frac12-\epsilon} \|f\|_{L^p(\R^d)}
\]
with $C_\epsilon$ uniform for $0<\delta<\frac14$ and $\varphi\in \CC^{N+1}(I)$, since $j_\circ =O(\log (1/\delta))$.

We now tern to the proof of \eqref{I0}. We first estimate $\CI_0$. To do this,
let us define the Littlewood-Paley projection operator $P_m$, $m\in \Z$, by
\begin{equation*}\label{LW}
\widehat{P_m f}(\xi)=\beta(2^{-m}|\xi|)\widehat{f} (\xi),
\end{equation*}
where $\beta$ is a smooth cutoff function supported in the interval $[\frac12,2]$ and satisfying $0\le \beta\le 1$ and $ \sum_{m\in \Z}\beta(2^{-m}t)=1$ for $t>0$. Since $\supp\, \varphi\subset [-1,1]$, $\rho \in [1/2,1]$, $\lambda\in [1,2]$, and $0<\delta\le 1/4,$
we see that
\[ \varphi\Big(\frac{\rho-|2^k\lambda\xi|^2}{\delta}\Big)\beta(2^{-m}\xi)\equiv 0\quad \mbox{ except } -3\le m+k\le 2. \] 
Using this we see that for any $k\in \Z$
\begin{equation}\label{eq} 
S^\varphi_{\rho,\delta, 2^k\lambda} f\,\,= \sum_{\substack{m\in\Z\,:\\ -3\le m+k\le 2}} S^\varphi_{\rho,\delta, 2^k\lambda} (P_mf).
\end{equation}
Note that  for each $k$ the number of non-vanishing $m$ is at most $6$. Thus, inserting \eqref{eq} into $\CI_0$ and applying the second relation in \eqref{scaling}, H\"older's inequality, and $l^p\subset l^\infty$, $(\CI_0)^p$ is bounded by a constant multiple of 
\begin{equation}\label{tem:prop}
 \sum_{k\in \Z}\sum_{\substack{m\in\Z\,:\\-3\le m+k\le 2}}\Big(\int_1^2  (2^k\lambda)^{2d/p}\Big\|\Big(\sum_{\rho\in \delta\Z\cap [1/2,1]} |S^\varphi_{\rho,\delta,1}(P_mf)_{2^k\lambda}|^2\Big)^{1/2}\Big\|_{L^p(\R^d)}^2 d\lambda \Big)^{p/2}.
\end{equation}
It was shown in \cite[Lemma 2.3]{JLV} that $L^p$-boundedness properties of the square function in \eqref{tem:sq} and the discretize square function in the above are essentially equivalent when $p\ge 2$. Hence by the assumption \eqref{tem:sq} we see that 
\begin{align*}
 \Big\|\Big(\sum_{\rho\in \delta\Z\cap [1/2,1]} |S^\varphi_{\rho,\delta,1}(P_mf)_{2^k\lambda}|^2\Big)^{1/2}\Big\|_{L^p(\R^d)}&\lesssim \delta^{-\beta-\frac12} \|(P_mf)_{2^k\lambda}\|_{L^p(\R^d)}\\
 &\lesssim \delta^{-\beta-\frac12} (2^k\lambda)^{ -d/p}\|P_mf\|_{L^p(\R^d)}
\end{align*}
holds with the implicit constant independent of  $\varphi\in \CC^{N+1}(I)$ and $0<\delta\le 1/4.$ Putting this back into \eqref{tem:prop}, we have
\begin{align*}
\CI_0 &\lesssim  \delta^{-\beta-\frac12} \Big(\sum_{k\in \Z}\sum_{\substack{m\in\Z\,:\\-3\le m+k\le 2}}\|P_mf \|_{L^p(\R^d)}^p \Big)^{1/p}\\
&\lesssim  \delta^{-\beta-\frac12} \|(\sum_{m\in\Z} |P_m f|^2)^{1/2}\|_{L^p(\R^d)} \lesssim    \delta^{-\beta-\frac12} \|f\|_{L^p(\R^d)},
\end{align*}
since $p\ge 2$. For the last inequality we use the Littlewood-Paley inequality.

It remains to estimate the term $\CI\CI$, which is much simpler. Since $S^\varphi_{\rho,\delta,\lambda}$ is a multiplier operator, this can be written  as $S^\varphi_{\rho,\delta,\lambda}f(x)=\lambda^{-d} K_{\rho,\delta}(\lambda^{-1} \cdot )\ast f(x)$, where
\[ K_{\rho,\delta}(x) =\int_{\R^d} e^{2\pi i x\cdot \xi} \varphi\big(\frac{\rho-|\xi|^2}\delta \big)d\xi.\]
From integration by parts, we see that $|K_{\rho,\delta}(x)|\le C_{\rho,\delta,\varphi}\, \delta^{d/2}(1+\delta^{1/2}|x|)^{-d-1}$. Especially, when $\rho\le 4\delta$, the constant $C_{\rho,\delta,\varphi}$ depends only on the $C^{d+1}$-norm of $\varphi$. Hence the constant is independent of  the choice of $\varphi\in \CC^{N+1}(I)$, $\rho\le 4\delta$, and $0<\delta\le 1/4$ whenever $N\ge d.$ Applying the kernel estimate in the above,
\begin{equation}\label{small}
\begin{aligned}
 |S^\varphi_{\rho,\delta,2^k\lambda}f(x)|
&\lesssim (2^k\lambda \delta^{-1/2})^{-d} \int_{\R^d}\big(1+2^{-k}\lambda^{-1} \delta^{1/2}|x-y|\big)^{-d-1}|f(y)|\,dy\\
&\lesssim \sum_{j=0}^{\infty} 2^{-j} \big(2^k\lambda \delta^{-1/2}2^{j}\big)^{-d} \int_{B^d(x, 2^{k}\lambda \delta^{-1/2}2^j)} |f(y)| \,dy\\
&\lesssim \sum_{j=0}^{\infty} 2^{-j} Mf(x) \lesssim Mf(x)
\end{aligned}
\end{equation}
holds uniformly for $\varphi \in\CC^{N+1}(I) $, $0\le\rho\le 4\delta$, and $k\in \Z$. 
Here $M$ is the Hardy-Littlewood maximal function.
Inserting this into $\CI\CI$, we get
\begin{align*} 
\CI\CI\lesssim \| Mf \|_{L^p(\R^d)}\lesssim \|f\|_{L^p(\R^d)},
\end{align*}
which completes the proof. 
\end{proof}

\section{Proof of Proposition \ref{prop:Lp}}
To verify Proposition  \ref{prop:Lp}, we adopt the  idea 
in \cite{JLV} which decomposes the bilinear operator  into a sum of 
products of linear operators. The decomposition lemma (Lemma 3.1 in \cite{JLV}) reduces  the problem of obtaining  $L^p\times L^q\to L^r$ estimates  for the auxiliary bilinear operator $\FB^\delta_1$ to that for a sum of products of two linear multiplier operators. In what follows we show that the argument also works for the maximal estimate.
We reformulate the decomposition argument as a single lemma (Lemma \ref{lem:decom} below)  which was implicit in \cite[Section 3]{JLV}.  This  provides  a pointwise bound on   the auxiliary operator $\FB^\delta_\lambda$  by  a sum of product of two linear operators $S^\phi_{\rho,\delta,\lambda}$ and $S^\phi_{\varrho-\rho,\delta,\lambda}$.  Since the proof of Lemma  \ref{lem:decom} is already contained in \cite{JLV} we only provide a brief of  sketch.  

\begin{lemma}\label{lem:decom} \cite{JLV}
Let $d\ge 2$, $N\in \mathbb N,$ $0<\delta\le 1/4,$ and  $\epsilon>0$. Set $\tilde{\delta}=\delta^{1+\epsilon}$. Then there exists $C=C(\psi, N)$  such that for any $\lambda>0$
\begin{align*}
&\left|\FB^\delta_\lambda(f_1,f_2)(x)\right| \le  C\times \\
&\sum_{\varrho\in \tilde{\delta}\Z\cap [1-4\delta,1+2\delta]} \sum_{\rho\in \tilde{\delta}\Z\cap [0,2]} 
\Big( \sum_{\substack{a,b\in \mathbb N\cup\{0\}\,:\\0\le a+b\le N}} C_{a,b}\delta^{\epsilon(a+b)} \left|S^{\varphi_a}_{\rho,\tilde{\delta},\lambda}f_1(x)S^{\varphi_b}_{\varrho-\rho,\tilde{\delta},\lambda}f_2(x)\right|\\
&\hspace{3cm}+ \delta^{\epsilon N} \int_\R \left|\wh\psi(\tau)\,T_{m_{\varrho,\rho}^{\tilde{\delta}}(\lambda\cdot,\lambda\cdot,\tau)}(f_1,f_2)(x) \right|d\tau\Big),
\end{align*}
where, $\varphi_a$ and $\varphi_b$ are in $\CC^N(I)$ for $0\le a,b\le N$, and
\[ T_{m_{\varrho,\rho}^{\tilde{\delta}}(\lambda\cdot,\lambda\cdot,\tau)}(f_1,f_2)(x) =\int_{\R^{2d}} e^{2\pi i x\cdot(\xi+\eta)} m_{\varrho,\rho}^{\tilde{\delta}}(\lambda \xi,\lambda\eta,\tau) \wh f_1(\xi) \wh f_2(\eta) d\xi g\eta.\]
Here $m_{\varrho,\rho}^{\tilde{\delta}}$ satisfies, for all multi-indices $\beta$ and $\gamma$, $|\beta|+|\gamma|\le N,$
\begin{equation}\label{re}
|\partial^\beta_\xi \partial^\gamma_\eta m_{\varrho,\rho}^{\tilde{\delta}}(\xi,\eta,\tau)|\le C (1+|\tau|)^N\tilde{\delta}^{-|\beta|-|\gamma|}
\end{equation}
with $C$ independent of $\rho$ and $\varrho$.
 \end{lemma}
 
\begin{proof}[Sketch of the proof of Lemma \ref{lem:decom}] 
Let $\varphi$ be a smooth function supported in $[-1,1]$ and satisfying $\sum_{k\in \Z}\varphi(t+k)=1$ for all $t\in \R.$ Then we can write
\begin{equation}\label{sum}
\begin{aligned}
 \FB^\delta_\lambda&(f_1,f_2)(x)  =
\\  & \sum_{\varrho\in \tilde{\delta}\Z\cap [1-4\delta,1+2\delta]} \sum_{\rho\in \tilde{\delta}\Z\cap [0,2]} 
\int_{\R^d\times\R^d} e^{2\pi i x\cdot(\xi+\eta)} \psi\Big(\frac{1-|\lambda\xi|^2-|\lambda\eta|^2}\delta\Big)
\\&\hspace{1cm}\times\varphi\Big(\frac{\rho-|\lambda\xi|^2}{\tilde{\delta}}\Big)\varphi\Big(\frac{\varrho-\rho-|\lambda\eta|^2}{\tilde{\delta}}\Big)\wh{f_1}(\xi)\wh f_2(\eta)d\xi d\eta,
\end{aligned}
\end{equation}
since $ \psi$ is supported in $[1/2,2].$ By the inversion formula, the multiplier of $\FB^\delta_\lambda$ is expressed by
\[ \psi\Big(\frac{1-|\lambda\xi|^2-|\lambda\eta|^2}\delta\Big) =\int_\R \wh \psi(\tau) e^{2\pi i\tau(\frac{\varrho-|\lambda\xi|^2-|\lambda\eta|^2}{\delta})}e^{2\pi i \tau(\frac{1-\varrho}\delta)}d\tau.\]
Applying Taylor's theorem for $e^{2\pi i\tau(\frac{\varrho-|\lambda\xi|^2-|\lambda\eta|^2}{\delta})}$, we have
\begin{align*}
e^{2\pi i\tau(\frac{\varrho-|\lambda\xi|^2-|\lambda\eta|^2}{\delta})} =\sum_{0\le a+b\le N} &c_{a,b} \left(\frac{\tau(\rho-|\lambda\xi|^2)}\delta\right)^a \left(\frac{\tau(\varrho-\rho-|\lambda\eta|^2)}\delta\right)^b \\
&+\Big(\begin{array}{c} {\text{remainder}}\\{\text{term}}\end{array}\Big).
\end{align*}
Inserting this into the first expression \eqref{sum} and properly arranging the involved terms, we get the desired decomposition.  In fact, 
the terms in the sum $\sum_{0\le a+b\le N} $  give rise to $S^{\varphi_a}_{\rho,\tilde{\delta},\lambda}f_1 S^{\varphi_b}_{\varrho-\rho,\tilde{\delta},\lambda}f_2 $ and  the remainder term   to $T_{m_{\varrho,\rho}^{\tilde{\delta}}(\lambda\cdot,\lambda\cdot,\tau)}(f_1,f_2)$. 
We refer to \cite[Proof of Lemma 3.1]{JLV} for the details.
\end{proof}

Once we have Lemma \ref{lem:decom} and the square function estimates in Section \ref{sec:w}, proof of Proposition  \ref{prop:Lp} is rather routine. 

\subsection*{Proof of Proposition \ref{prop:Lp}} Fix $\epsilon>0$ and set $\widetilde{\delta}=\delta^{1+\epsilon}.$  Thanks to Lemma \ref{lem:decom}, it suffices to obtain,  for any $\varrho\in \widetilde\delta\Z\cap[1/2,2]$ and  $0\le a,b\le \CN,$
the desired bound on
\begin{equation}\label{main}
\Big\| \sup_{k\in\Z}\int_1^2\sum_{\rho\in \widetilde{\delta}\Z\cap[0,2]}\big| S^{\varphi_a}_{\rho,\widetilde{\delta},2^k\lambda}f_1(x)S^{\varphi_b}_{\varrho-\rho,\widetilde{\delta},2^k\lambda}f_2(x)\big| d\lambda \Big\|_{L^r(\R^d)}
\end{equation}
and
\begin{equation}\label{error}
\Big\| \sup_{k\in\Z}\int_1^2\int_\R \big|\wh\psi(\tau) \,T_{m_{\varrho,\rho}^{\widetilde{\delta}}(2^k \lambda\cdot,2^k \lambda\cdot,\tau)} (f_1,f_2)(x) \big|d\tau d\lambda \Big\|_{L^r(\R^d)},
\end{equation}
where $\CN$ is a large integer which will be chosen later, $\varphi_a$ and $\varphi_b$ are in $ \CC^\CN(I)$, and $m_{\varrho,\rho}^{\widetilde{\delta}}$ satisfies the estimate \eqref{re}.

Since  $\varphi_a,\varphi_b\in \CC^\CN(I)$ for all $0\le a,b\le \CN$, Proposition \ref{prop:square} and Corollary \ref{lem:L2} give the estimate
\begin{equation}\label{e_main} 
\eqref{main}\,\lesssim_\epsilon\, \widetilde{\delta}^{-\alpha_{\fp}(p,q)+1-\epsilon}  \|f_1\|_{L^p(\R^d)}  \|f_2 \|_{L^q(\R^d)}
\end{equation}
whenever $\CN$ is large enough. Here the implicit constant is independent of  $\varrho$ and $0\le a,b\le \CN$. More precisely, by H\"older's inequality, we have 
\[ \eqref{main} \le \prod_{j=1,2}\Big\| \Big(\sup_{k\in\Z}\int_1^2 \left|\FD_j(f_j)\right|^2 d\lambda \Big)^{1/2}\Big\|_{L^{p_j}(\R^d)},\]
where $p_1=p$, $p_2=q$, $\FD_1(f_1)=\big(\sum_{\rho\in \widetilde{\delta}\Z\cap[0,2]}\big| S^{\varphi_a}_{\rho,\widetilde{\delta},2^k\lambda}f_1\big|^2\big)^{1/2}$, and $\FD_2(f_2)=\big(\sum_{\rho\in \widetilde{\delta}\Z\cap[0,2]}\big| S^{\varphi_b}_{\varrho-\rho,\widetilde{\delta},2^k\lambda}f_2\big|^2\big)^{1/2}$. Note that $\FD_1=\FD^{\varphi_a}_{\widetilde\delta,k}$ and $\FD_2 \le \FD^{\varphi_b}_{\widetilde\delta,k}$, since $\varrho\in \widetilde{\delta}\Z$ and $S^{\varphi_a}_{\rho,\widetilde\delta}f\equiv 0$ whenever $\rho\in \widetilde\delta\Z$ is negative. Moreover, by interpolation between the estimates in  Proposition \ref{prop:square} and Corollary \ref{lem:L2}, we see that  $\FD^{\varphi_a}_{\widetilde{\delta},k}$ satisfies
\[ \big\|\sup_{k\in\Z}|\FD^{\varphi_a}_{\widetilde{\delta},k}|\big\|_{L^p(\R^d)}\lesssim 
\begin{cases}
\widetilde\delta^{-\alpha(p)-\epsilon}&\text{if }\fp<p\le \infty,\\
\widetilde\delta^{-\alpha(\fp)(1-\frac2p)/(1-\frac2\fp)-\epsilon}&\text{if }2\le p\le \fp,
\end{cases}
 \]
if $\CN > \max \{N_\circ, d+1\}$. This yields the estimate \eqref{e_main}.

 We now consider the term \eqref{error}. Since $T_{m_{\varrho,\rho}^{\widetilde{\delta}}(2^k \lambda \cdot,\tau)} $ is a bilinear multiplier operator, 
we may write \eqref{error} as 
\[ \Big\| \sup_{k\in\Z}\int_1^2\int_\R  \Big|\wh\psi(\tau)\int_{\R^{2d}} K_{\varrho,\rho}^{\widetilde{\delta}}(x-y,x-z) f_1(y) f_2(z) dy dz \Big|d\tau d\lambda \,\Big\|_{L^r(\R^d)},\]
where $K^{\widetilde{\delta}}_{\varrho,\rho}(y,z)=\CF^{-1}(m_{\varrho,\rho}^{\widetilde{\delta}} (2^k\lambda\cdot,2^k\lambda\cdot,\tau))(y,z) $. Notice that  $m^{\widetilde{\delta}}_{\varrho,\rho}$ satisfies \eqref{re}. Thus, by integration by parts, it is easy to see that  the absolute value of the kernel $K^{\widetilde\delta}_{\varrho,\rho}$  is bounded  by
\[C(1+|\tau|)^N(2^{k}\lambda)^{-2d}(1+2^{-k}\lambda^{-1}\widetilde{\delta}|y|)^{-d-\frac12}(1+2^{-k}\lambda^{-1}\widetilde{\delta}|z|)^{-d-\frac12}\]
with $C>0$ independent of $\rho,\varrho,$ and $\widetilde\delta$,  if $\CN>2d.$ Similarly as in  \eqref{small}, by  H\"older's inequality  we see
\begin{align*}
\eqref{error}&\lesssim \widetilde{\delta}^{-2d} \Big\| \int_1^2\int_\R  |\wh\psi(\tau)|(1+|\tau|)^N Mf_1(x) Mf_2(x) d\tau d\lambda \,\Big\|_{L^r(\R^d)}\\
&\lesssim  \widetilde{\delta}^{-2d}\|Mf_1 \|_{L^p(\R^d)}\|Mf_2 \|_{L^q(\R^d)}\lesssim \tilde{\delta}^{-2d} \|f_1\|_{L^p(\R^d)} \|f_2\|_{L^q(\R^d)}.
\end{align*}
Here the second inequality holds because of $\psi \in \CS(\R).$

Combining all of the above estimates, we have for any  $\epsilon>0$
\begin{align*}
&\Big\| \sup_{k\in\Z} \int_1^2|\FB^\delta_{2^k\lambda}(f_1,f_2)(x)|d\lambda \Big\|_{L^r(\R^d)}\\
\lesssim & \, \delta^{-\alpha_{\fp}(p,q)+1 }\big(\delta^{-\epsilon\alpha_\fp(p,q)} +\delta^{\alpha_{\fp}(p,q)-2d +\epsilon(\CN-2d-2)}\big) \|f_1\|_{L^p(\R^d)}\|f_2\|_{L^q(\R^d)}.
\end{align*}
 By choosing a sufficiently large $\CN$ to be $ \alpha_{\fp}(p,q)-2d +\epsilon(\CN-2d-2)>0$, we obtain the desired estimate \eqref{eq:Lp}.
This completes the proof. 

{\bf Acknowledgement.} This work was  supported by the POSCO Science Fellowship of POSCO TJ Park Foundation (E. Jeong) and  by NRF-2018R1A2B2006298 (S. Lee).  
The authors would like thank D. He and L. Yan for discussion at early stage of this project and Y. Heo for making  the preprint \cite{HHY} available. 

\bibliographystyle{amsalpha}

\end{document}